\documentclass[11pt, reqno,
]{amsart}
\usepackage{a4wide}
\usepackage[T1]{fontenc}
\usepackage[utf8]{inputenc}
\usepackage{url}
\usepackage{placeins}
\usepackage{xcolor}
\usepackage{graphicx}
\usepackage[font=small,labelfont=bf]{caption}
\usepackage{subcaption}
\usepackage{float}
\usepackage{amsmath}
\usepackage{amssymb}
\usepackage{mathtools}
\usepackage{amsthm} 
\usepackage{bbm}
\usepackage{upgreek}
\usepackage{enumitem}
\setenumerate[1]{label=(\roman*)}

\numberwithin{equation}{section}

\usepackage[
  hypertexnames=false, %to make autonum work
  colorlinks,
  breaklinks,
  unicode
]{hyperref}
\hypersetup{
    colorlinks = true,
    urlcolor = blauD,
    linkcolor = blauD,
    citecolor = blauD
}
\usepackage[noabbrev, capitalise]{cleveref} 
 \newtheorem{thm}{Theorem}[section]
 
 \newtheorem{lem}[thm]{Lemma}
\newtheorem{prop}[thm]{Proposition}
 
\theoremstyle{definition}
 
  \newtheorem{rem}[thm]{Remark}
  \newtheorem{ex}[thm]{Example}

\makeatletter % to repeat theorem numbers
\newtheorem*{rep@theorem}{\rep@title}\newcommand{\newreptheorem}[2]{%
\newenvironment{rep#1}[1]{%
\def\rep@title{\bf #2 \ref{##1}}%
\begin{rep@theorem}}%
{\end{rep@theorem}}}
\makeatother
\newreptheorem{theorem}{Theorem}

\usepackage{autonum}

\definecolor{turkisD}{HTML}{579C87}
\definecolor{turkisH}{HTML}{A1C181}
\definecolor{gelb}{HTML}{FCCA46}
\definecolor{orange}{HTML}{FE7F2D}
\definecolor{blauD}{HTML}{233D4D }

\newcommand{\demph}[1]{\textbf{#1}} %\emph{#1}

\DeclarePairedDelimiter{\abs}{\lvert}{\rvert}

\DeclareMathOperator{\mult}{mult}
\DeclareMathOperator{\conv}{conv}
\DeclareMathOperator{\spn}{span}
\DeclareMathOperator{\cone}{cone}

\DeclareMathOperator{\co}{co}
\DeclareMathOperator{\codim}{codim}
\DeclareMathOperator{\Oc}{in}
\DeclareMathOperator{\ehr}{ehr}
\DeclareMathOperator{\Ex}{cu}

\newcommand{\Z}{\mathbbmss{Z}}
\newcommand{\R}{\mathbbmss{R}}
\newcommand{\N}{\mathcal{N}}
\newcommand{\M}{\mathcal{M}}
\newcommand{\Hy}{\mathcal{H}}

\newcommand{\Or}{\upvarsigma}
\newcommand{\Pol}{\mathcal{P}}
\newcommand{\Po}{\mathsf{P}}
\newcommand{\Q}{\mathcal{Q}}
\newcommand{\B}{\mathcal{B}}
\newcommand{\va}{\mathbf{a}}
\newcommand{\x}{\mathbf{x}}
\newcommand{\y}{{y}}
\newcommand{\vt}{{t}}

\title%
[Pruned inside-out polytopes]%
{Pruned inside-out polytopes, combinatorial reciprocity theorems and generalized permutahedra}

\makeatletter
\@namedef{subjclassname@2020}{%
  \textup{2020} Mathematics Subject Classification}
\makeatother

\subjclass[2020]{52Bxx; 52B20; 52C35; 05Axx; 05C15; 05C65}

\author{Sophie Rehberg}
\date{June 13, 2022}
\keywords{inside-out polytopes, Ehrhart polynomials, Ehrhart-Macdonald reciprocity, combinatorial reciprocity theorems,  hypergraphic polytopes, hypergraph colorings and orientations}

\begin{document}

\begin{abstract}

Generalized permutahedra are a class of polytopes with many interesting combinatorial subclasses. 
We introduce pruned inside-out polytopes, a generalization of inside-out polytopes introduced by Beck--Zaslavsky (2006), which have many applications such as recovering the famous reciprocity result for graph colorings by Stanley.
We study the integer point  count of pruned inside-out polytopes by applying classical Ehrhart polynomials and Ehrhart--Macdonald reciprocity.
This yields a geometric perspective on and a generalization of a combinatorial reciprocity theorem for generalized permutahedra by Aguiar--Ardila (2017), Billera--Jia--Reiner (2009), and Karaboghossian (2022). 
Applying this reciprocity theorem to hypergraphic polytopes allows to give a geometric proof of a combinatorial reciprocity theorem for hypergraph colorings by Aval--Karaboghossian--Tanasa (2020). 
This proof relies, aside from the reciprocity for generalized permutahedra, only on elementary geometric and combinatorial properties of hypergraphs and their associated polytopes.
 
\end{abstract}

\maketitle
\tableofcontents
% \newpage

\section{Introduction}

Generalized permutahedra are an interesting class of polytopes containing numerous subclasses of polytopes defined via combinatorial structures,
such as graphic zonotopes, hypergraphic polytopes (Minkowski sums of simplices), simplicial complex polytopes, matroid polytopes, associahedra, and nestohedra.
Generalized permutahedra themselves are closely related to submodular functions, which have applications in optimization.

A \emph{combinatorial reciprocity theorem} can be described as a result that relates two classes of combinatorial objects via their enumeration problems (see, e.g., \cite{Stanley1974, Beck2018}).
For example, the number  of proper $m$-colorings of a graph $g=(I,E)$ agrees with a polynomial $\chi(g)(m) $ of degree $d=\lvert I\rvert$ for positive integers $m\in\Z_{>0}$, and $(-1)^d\chi(g)(-m)$ counts the number of  pairs of compatible acyclic orientations and $m$-colorings of the graph $g$  \cite{Stanley1973}.
For precise definitions see \cref{Assec:hg} below.

One of our main results is a combinatorial reciprocity theorem for generalized permutahedra counting integral directions with $k$-dimensional maximal faces:
\begin{reptheorem}{thm:moredim}
 For a generalized permutahedron $\Pol\subset \R^d$ and $k=0,\dots,d-1$,
\begin{equation}
\begin{split}
 \chi_{d,k}(\Pol)(m) \coloneqq& \ \#{\big\{\y\in [m]^d \ \colon\ \y\text{-maximum face $\Pol_\y$ is a $k$-face}\big\}}
\end{split}
\end{equation}
 agrees with a polynomial of degree $d-k$, and
\begin{equation}
\begin{split}
 (-1)^{d-k}\chi_{d,k}(\Pol)(-m) &= \sum_{\y\in[m]^d} 
			 \# \left(k\text{-faces of } \Pol_\y\right).%\\
\end{split}
\end{equation}
\end{reptheorem}
\noindent
We will use integer point counting in dissected and dilated cubes to prove this result
and comment on further generalizations in \Cref{rem:beyond}.

The special case of this theorem for $k=0$, i.e., generic directions, was obtained by Aguiar and Ardila \cite{AA17}, and earlier by Billera, Jia, and Reiner \cite{Billera2009} in a slightly different language.
The $k=0$ case  was also recently extended in \cite{karaboghossian_combinatorial_2022}.
As shown for some examples in \cite[Section~18]{AA17}  the application of such a result to the various subclasses of generalized permutahedra yields already known combinatorial reciprocity theorems for their related combinatorial structures  such as matroid polynomials \cite{Billera2009}, Bergmann polynomials of matroids and Stanley's famous reciprocity theorem for graph colorings \cite{Stanley1973}.

Aguiar and Ardila develop a Hopf monoid structure on the species of generalized permutahedra, work with polynomial invariants defined by characters, and apply their antipode formula to get the combinatorial interpretation of the reciprocity result for generalized permutahedra for $k=0$ (\cref{cor:genP}, below) \cite[Sections 16, 17]{AA17}. This method is also used in \cite{karaboghossian_combinatorial_2022}.
The approach in \cite{Billera2009} is similar to the one by Aguiar and Ardila. Billera, Jia, and Reiner use  Hopf algebras of matroids and quasisymmetric functions, as well as a multivariate generating function as isomorphism invariants of matroids.
The reciprocity providing ingredient is again the antipode of a Hopf algebra together with Stanley's reciprocity for $P$-partitions \cite[Sections~6 and 9]{Billera2009}.

We give a different, geometric perspective.
In order to prove \cref{thm:moredim} we apply Ehrhart--Macdonald reciprocity to pruned inside-out polytopes. 
A pruned inside-out polytope ${\Q\setminus \bigcup\N^{\co1}}$  consist of the points that lie inside a polytope $\Q$ but not in the codimension one cones $\N^{\co1}$ of a complete polyhedral fan $\N$.
This is a generalization of inside-out polytopes introduced by Beck and Zaslavsky \cite{Beck2006}.
An inside-out polytope $\Q\setminus\Hy$ consists of the points in a polytope $\Q$ but off the hyperplanes in the arrangement $\Hy$. 
We think of the codimension-one cones $\N^{\co1}$ defining a pruned inside-out polytope as \emph{pruned} hyperplanes, hence the name.
One of the many applications of inside-out polytopes \cite{Beck2006a,Beck2006b,Beck2010,Beck2018} is yet a different proof of Stanley's reciprocity result for graph colorings \cite{Stanley1973}. 

 Aval, Karaboghossian, and Tanasa presented a reciprocity theorem for hypergraph colorings \cite{Aval2020}, generalizing Stanley's result for graph colorings. 
A main tool in the paper is a Hopf monoid structure on hypergraphs defined in \cite[Section 20.1.]{AA17} and the associated basic polynomial invariant.
However, they do \emph{not} use the antipode as reciprocity inducing element, but rather technical computations involving Bernoulli numbers.

 In \cref{Assec:hg} we show how the reciprocity theorem for hypergraph colorings in \cite{Aval2020} is a consequence of the reciprocity for generalized permutahedra. 
 Our main tool is a vertex description of hypergraphic polytopes in terms of acyclic orientations of hypergraphs (\cref{prop:hgPconv}).
More recent work by Karaboghossian \cite{karaboghossian_polynomial_2020, karaboghossian_combinatorial_2022} presents a more general version of the combinatorial reciprocity result for hypergraphs and an alternative proof with similar techniques as we present in \cref{Assec:hg}.

 As spelled out in \cite[Sections~21--25]{AA17} and \cite[Section~4]{Aval2020} hypergraphs and hypergraphic polytopes contain a number of interesting combinatorial subclasses such as simple hypergraphs, graphs, simplicial complexes, building sets, set partitions, and paths, together with their associated polytopes such as graphical zonotopes, simplicial complex polytopes, nestohedra, and graph associahedra.

The paper is organized as follows:
In \cref{sec:prunedEhr} we introduce the notion of pruned inside-out polytopes, define two counting functions on pruned inside-out polytopes, and derive \mbox{(quasi-)polynomiality} and reciprocity results.
\cref{sec:appl} provides two applications of the results in \cref{sec:prunedEhr}; 
first, to generalized permutahedra, giving a new geometric perspective on reciprocity theorems in \cite{Billera2009,AA17,karaboghossian_combinatorial_2022} and, moreover,
presenting  generalized versions for arbitrary face dimensions (\cref{ssec:AgenP}).
The relationship between our approach and the polynomial invariants for Hopf monoids is analyzed in \cref{ssec:relation}.
Secondly, we apply the reciprocity theorem for generalized permutahedra to the subclass of hypergraphic polytopes giving an elementary combinatorial and geometric proof of the reciprocity theorem for hypergraph colorings in \cite{Aval2020} (\cref{Assec:hg}). 
In \cref{appendix} we provide more details about generalized permutahedra. 
In particular we give a self-contained proof of the well known bijection between generalized permutahedra and submodular set functions. 
We do not claim the proof to be either new or original, but it is hard to find in the literature.

\section{Pruned inside-out polytopes and Ehrhart theory}\label{sec:prunedEhr}

 In \cite{Beck2006} Beck and Zaslavsky develop the notion of an inside-out polytope, that is, a polytope dissected by hyperplanes. 
 Counting integer point in a polytope but off certain hyperplanes turns out to be a useful tool to derive (quasi-)polynomiality results and reciprocity laws  for various applications such as  graph colorings 
 and signed graph colorings, composition of integers,  nowhere-zero flows on graphs and signed graphs, antimagic labellings, as well as magic, semimagic, and magic latin squares \cite{Beck2006a,Beck2006b,Beck2010}.
 After reviewing the necessary notions from polytopes and Ehrhart theory (\cref{ssec:polytopes&Ehrhart}),
 we introduce a generalization of inside-out polytopes, which we call pruned inside-out polytopes and develop  Ehrhart-theoretic results (\cref{ssec:piop}).

\subsection{Preliminaries: Polytopes and Ehrhart theory}\label{ssec:polytopes&Ehrhart}

First recall some basic notions from polytopes; for more detailed information consult, e.g., %TODO 
\cite{ziegler_lectures_1998, Gruenbaum2003}. 
A \demph{polyhedron} $\Pol\subset\R^d$ is the  intersection of finitely many halfspaces. 
If the intersection is bounded it is called a \demph{polytope} and can equivalently be described as the convex hull of finitely many points in $\R^d$.
A \demph{(polyhedral) cone} $N$ is a polyhedron such that for $x\in N$ the point $\lambda x$ is again contained in $N$ for every $\lambda\in\R_{\geq0}$.
A \demph{supporting hyperplane} $H$ of a polyhedron $\Pol$ is a hyperplane such that the polyhedron is contained in one of the closed halfspaces. 
The intersection of a polyhedron $\Pol$ with a supporting hyperplane $H$ is a \demph{face} $F=H\cap \Pol$  of $\Pol$. 
The \demph{dimension} $\dim(\Pol)$ (resp. $\dim(F)$) of a polyhedron $\Pol$ (resp. face $F$) is the dimension of the affine hull of the polytope $\Pol$ (resp. face $F$),
0-dimensional faces are called \demph{vertices} and $(\dim(\Pol)-1)$-dimensional faces are called \demph{facets}.
The \demph{codimension} $\codim(F)$ of an polyhedron $F$ is the difference between the dimension of the ambient space and the dimension of the polyhedron $\dim(F)$.
A polyhedron $\Pol$ is a \demph{rational polyhedron}, if all its facet defining hyperplanes $H$ can be described as $H=\left\{\x\in\R^d\,\colon\, \va\cdot \x=b\right\}$ for some $\va\in\Z^d$ and $b\in\Z$.

For a polytope $\Q\subset\R^d$ and a positive integer $t\in\Z_{> 0}$ we define the \demph{t\textsuperscript{th} dilate of} $\Q$ as
\begin{equation}
 t\Q\coloneqq \big\{x\in\R^d\,\colon\, \tfrac{1}{t} x\in\Q \big\}=\big\{tx\in\R^d\,\colon\, x\in\Q\big\}\,.
\end{equation}
The \demph{Ehrhart counting function} $\ehr_\Q(t)$ counts the number of integer point in the  
$t$\textsuperscript{th} dilate of the polytope $\Q$:
\begin{equation}
 \ehr_\Q(t)\coloneqq \#\left(\tfrac{1}{t}\Z^d\cap\Q\right)=\#\left(\Z^d\cap t\Q\right)\,.
\end{equation}
See \cref{fig:cube} for an example.
\begin{figure}
\centering
 \includegraphics[width=.5\textwidth]{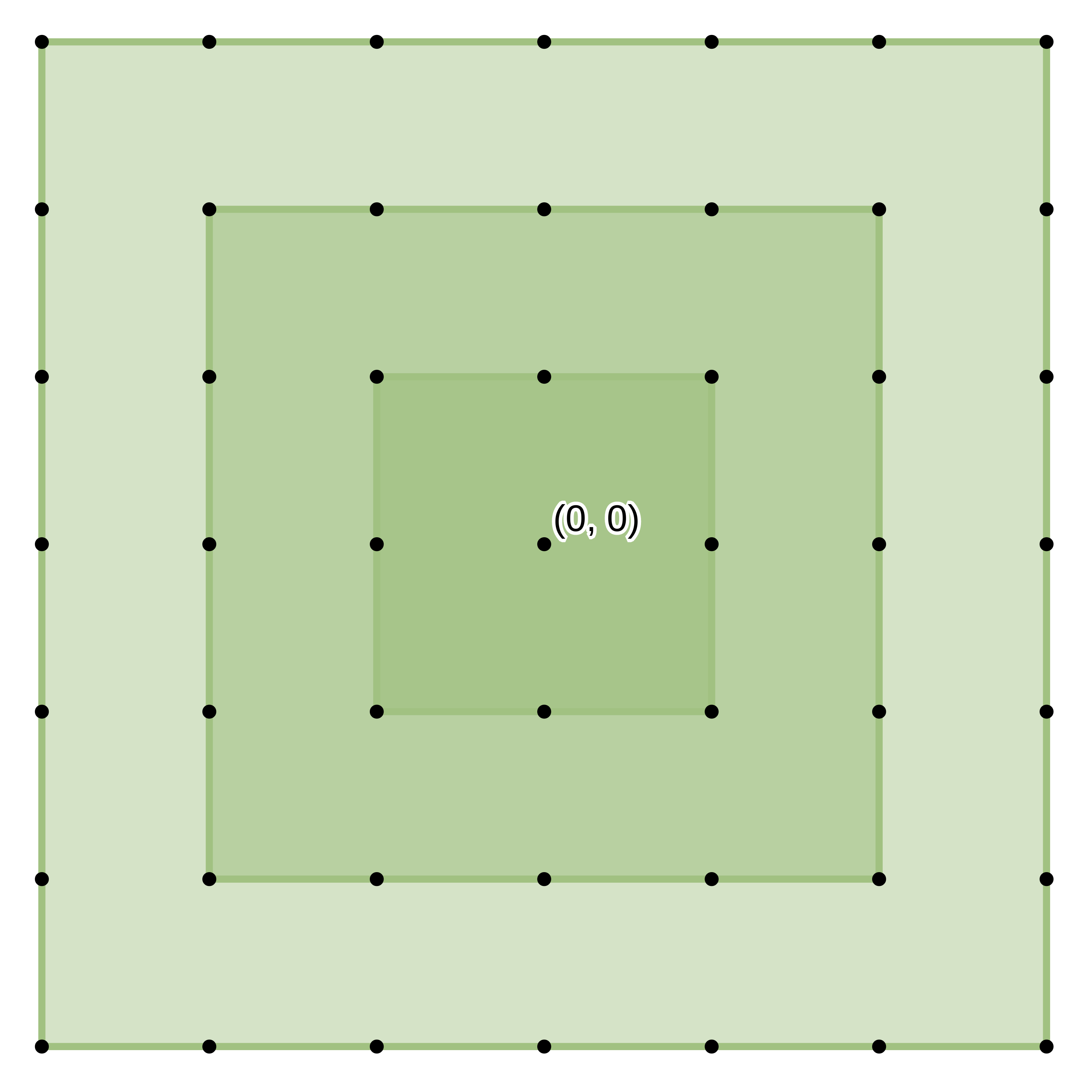}
 \caption{The cube $[-1,1]^2$ and its dilates $[-2,2]^2$ and $[-3,3]^2$ with Ehrhart  function $\ehr_{[-1,1]^2}(t)=(2t+1)^2$ and $\ehr_{(-1,1)^2}(t)=(2t-1)^2$.}\label{fig:cube}
\end{figure}
Recall that a \demph{rational (resp. integral) polytope} has vertices with rational (resp. integral) coordinates.
We call the least common multiple of the denominators of all coordinates of all vertices of a rational polytope the \demph{denominator of $\Q$.}
A \demph{quasipolynomial} of degree $d$ is a function $q\colon \Z\to\R$
 of the form
$q(t) = c_d(t) t^d + \dots + c_1(t) t + c_0(t)$ where 
$c_0, c_1,\dots, c_d\colon \Z\to\R$ are periodic functions.
The least common period of $c_0(n), c_1(n),
\dots, c_d(n)$ is the \demph{period} of $q(t)$.
\begin{thm}[Ehrhart's theorem {\cite{ehrhart_sur_1962}}]\label{thm:ehr} 
For a rational polytope $\Q\subset\R^d$ the Ehrhart counting function $\ehr_\Q(t)$ agrees with a  quasipolynomial of degree equal to the dimension of $\Q$ and period dividing the denominator of $\Q$ for all $t\in\Z_{>0}$.
\end{thm}
For an integer polytope $\Q$ Ehrhart's theorem implies that the Ehrhart counting function $\ehr_\Q$ is a polynomial. 
Therefore it is often called the \demph{Ehrhart polynomial}.
The following reciprocity theorem was conjectured and proved for various special cases by Eugéne Ehrhart and proved by Ian G. Macdonald.
It is the foundation for the results in this paper. 
\begin{thm}[Ehrhart--Macdonald reciprocity {\cite{macdonald_polynomials_1971}}]\label{thm:EhrMac}
Let $\Q\subset\R^d$ be a rational polytope and $t\in\Z_{>0}$. Then 
\begin{equation}
 (-1)^{\dim \Q} \ehr_\Q(-t)=\ehr_{\Q^\circ}(t)\coloneqq \#\left(\Z^d\cap t\Q^\circ\right)
\end{equation}
where $\Q^\circ$ is the (relative) interior of the polytope $\Q$.
\end{thm}

\subsection{Pruned inside-out polytopes and their counting functions}\label{ssec:piop}

Let $\N$ be a \demph{complete fan} in $\R^d$, that is, a family of polyhedral cones such that 
\begin{enumerate}[nolistsep]
 \item every non-empty face of a cone $N \in \N$ is also contained in $\N$,
 \item the intersection of two cones in $ \N$ is a face of both cones, 
 \item \label{iii} the union of the cones in the fan $\N$ covers the ambient space $\R^d$, i.e., 
 \begin{equation}
  \bigcup \N\coloneqq \bigcup_{N\in\N}N=\R^d \,.
 \end{equation}
\end{enumerate}
For an introduction to complete fans consult, e.g., \cite[Section 7.1]{ziegler_lectures_1998}. 
A fan is called \demph{rational} if its cones $N \in \N$ are generated by rational vectors.
For a complete fan $\N$ in $\R^d$ we define the \demph{codimension-one fan}%
\footnote{This is still a fan, but it is not complete anymore, i.e., condition \ref{iii} in the above definition is not fulfilled.}
 $\N^{\co 1}$ in $\R^d$ to contain the cones in $\N$ with codimension $\geq 1$, that is, all but the full-dimensional cones in $\N$:
  \begin{equation}
  \N^{\co 1}\coloneqq \big\{ N\in\N\,\colon\, \dim N\leq d-1 \big\}\,.\label{def:codim1fan}
 \end{equation}
 We think of the codimension-one fan as a \emph{pruned} hyperplane arrangement, since cones of codimension one can be seen as parts of hyperplanes.
\begin{figure}
\centering
\begin{subfigure}[t]{.23\textwidth}
  \centering
  \includegraphics[]{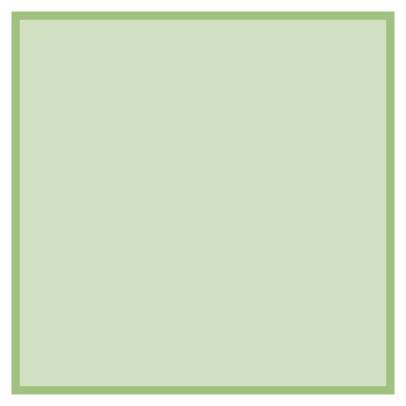}
\caption{Polytope $[-1,1]^2$.}\label{fig:Q}
\end{subfigure}\hfill%
\begin{subfigure}[t]{.23\textwidth}
  \centering
  \includegraphics[]{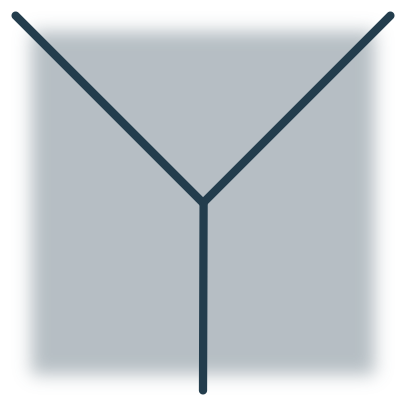}
\caption{Complete fan $\M$ and its codimension-$1$ fan $\M^{\co1}$ (rays) as defined in \Cref{ex:piop}.}\label{fig:N}
\end{subfigure}\hfill%
\begin{subfigure}[t]{.24\textwidth}
  \centering
  \includegraphics[]{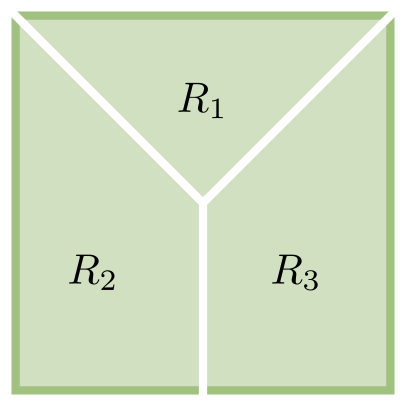}
\caption{The (closed) pruned inside-out poly\-tope $[-1,1]^2\setminus\bigcup\M^{\co1}$ with regions $R_1,R_2,R_3$. }\label{fig:Q-N}
\end{subfigure}\hfill%
\begin{subfigure}[t]{.23\textwidth}
  \centering
  \includegraphics[]{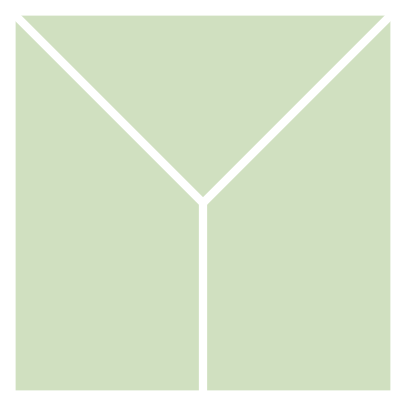}
\caption{The open pruned inside-out poly\-tope $(-1,1)^2\setminus\bigcup\M^{\co1}$ with open regions.}\label{fig:Qo-N}
\end{subfigure}
\caption{Construction of  pruned inside-out polytopes  and their regions as in \Cref{ex:piop}.}\label{abb:pruned}
\end{figure}
For a polytope $\Q\subset\R^d$ and a complete  fan $\N$ in $\R^d$  we call 
\begin{equation}
\begin{split}
  \Q\setminus \Big(\bigcup \N^{\co 1}\Big) 
  =& \biguplus_{\substack{N\in\N,\\ N \text{ full-dimensional}} } \left(\Q\cap N^\circ\right) 
\end{split}
\end{equation}
 a \demph{pruned inside-out polytope} and we call the connected components in $\Q\setminus \left(\bigcup \N^{\co 1}\right)$ \demph{regions}. 
 So, a  pruned inside-out polytope $\Q\setminus \left(\bigcup \N^{\co 1}\right)$ is the disjoint union of its regions $\Q\cap N^\circ$, where $N^\circ$ is an open full-dimensional cone in $\N$. 
 We will mostly consider open pruned inside-out polytopes  $\Q^\circ\setminus \left(\bigcup \N^{\co 1}\right)$, which decompose into disjoint open polytopes, the regions. 
  A pruned inside-out polytope is \demph{rational} if the topological closures of all its regions are rational polytopes.
 \begin{ex}\label{ex:piop}
 Let $[-1,1]^2\subseteq \R^2$ be a square (see \Cref{fig:Q}).
 Let $\M\coloneqq\{N_1,N_2,N_3,\dots\}$ be the complete fan consisting of all faces of the three full-dimensional cones
 \begin{equation}
 \begin{split}
%   N_1\coloneqq\cone\{(1,1),(-1,1)\}\,,\quad
%   N_2\coloneqq\cone\{(-1,1),(0,-1)\}\,,\quad
%   N_3\coloneqq\cone\{(0,-1), (1,1)\}\,
  N_1\coloneqq\{x\in\R^2\colon x_2+x_1\geq 0,\ x_2-x_1\geq0\}\,,\quad
  N_2&\coloneqq\{x\in\R^2\colon x_2\leq 0,\ x_2+x_1\leq0\}\,,\\
  N_3&\coloneqq\{x\in\R^2\colon x_2\geq 0,\ x_2-x_1\leq0\}\,.
 \end{split}
 \end{equation}
 Then the codimension-one fan $\M^{\co1}=\{n_1,n_2,n_3\}$ consists of the three rays
 \begin{equation}
  n_1\coloneqq\{ (\lambda,\lambda)\in\R^2\colon\lambda\geq0\}\,,\ 
  n_2\coloneqq\{ (-\lambda,\lambda)\in\R^2\colon\lambda\geq0\}\,,\ 
  n_3\coloneqq\{ (0,-\lambda)\in\R^2\colon\lambda\geq0\}\,.
 \end{equation}
 See \Cref{fig:N}.
 The pruned inside-out polytope
 \begin{equation}
  [-1,1]^2\setminus\bigcup\M^{\co1}=[-1,1]^2\setminus\bigcup_{i=1}^3 n_i = \bigcup_{i=1}^3\left( [-1,1]^2\cap N_i^\circ\right)
 \end{equation}
 is composed of three half-open regions $R_1,R_2,R_3$, see \Cref{fig:Q-N}. Their topological closures can be described as
 \begin{equation}
 \begin{split}
  \overline{R}_1=\conv\{(0,0),(1,1),(-1,1)\}\,,\quad
  \overline{R}_2&=\conv\{(0,0),(-1,1), (-1,-1),(0,1)\}\,,\\
  \overline{R}_3&=\conv\{(0,0),(0,1),(1,-1), (1,1)\}\,.
 \end{split}
 \end{equation}
 The open pruned inside-out polytope 
  \begin{equation}
  (-1,1)^2\setminus\bigcup\M^{\co1}=(-1,1)^2\setminus\bigcup_{i=1}^3 n_i = \bigcup_{i=1}^3 R_i^\circ
 \end{equation}
 is depicted in \Cref{fig:Qo-N}.
 \end{ex}
For a positive integer $t\in\Z_{>0}$ we define the \demph{inner pruned Ehrhart function} as
\begin{equation}
\begin{split}
  \Oc _{\Q,\N^{\co 1}}(t)\coloneqq& \#\bigg(\frac{1}{t}\Z^d\cap \Big(\Q\setminus \Big(\bigcup \N^{\co 1}\Big)\Big)\bigg)=\#\bigg(\Z^d\cap t\cdot \Big(\Q\setminus \Big(\bigcup \N^{\co 1}\Big) \Big)\bigg)\,,
\end{split}
\end{equation}
where
\begin{equation}
\begin{split}
 t\cdot\Big(\Q\setminus \Big(\bigcup \N^{\co 1}\Big) \Big)\coloneqq&
 \ t\cdot\Q\,\setminus\, \Big(\bigcup t\cdot\N^{\co 1}\Big) \\
 \coloneqq&\big\{t\y\in\R^d\,\colon\, \y\in\Q\big\}\setminus \big\{t\y\in\R^d\,\colon\, \y\in N, \text{ for some }N\in\N^{\co1}\big\} \,.
\end{split}
\end{equation}
See \cref{abb:InQ,abb:InQo} for illustrations.
\begin{lem}\label{lem:InSum}
 For a  polytope $\Q \subset \R^d$ and a complete fan $\N $ in $\R^d$,
 \begin{equation}
	\Oc_{\Q^\circ,\N^{\co 1}}(t)=\sum_{i=1}^k \ehr_{R^\circ_i}(t)
 \end{equation}
where $R^\circ_i$ are the open regions of the open pruned inside-out polytope $\Q^\circ\setminus \left(\bigcup \N^{\co 1}\right)$.
\end{lem}
 \begin{proof}
   We  decompose the pruned inside-out polytope $\Q\setminus \left(\bigcup \N^{\co 1}\right)$ into its regions $R_1,\dots,R_k$.
 Then the open pruned inside-out polytope $\Q^\circ\setminus \left(\bigcup \N^{\co 1}\right)=\biguplus_{i=1}^k R_i^\circ$ is the disjoint union of the open polytopes $R_1^\circ,\dots,R_k^\circ$.
 The result follows since counting lattice points is a valuation (see, e.g., \cite[Section 3.4]{Beck2018}).
 \end{proof}
Furthermore, we define a second counting function for pruned inside-out polytopes, the 
\demph{cumulative pruned Ehrhart function} $\Ex_{\Q,\N^{\co 1}}(\Z^d)$, for a positive integer $t\in\Z_{>0}$ as
\begin{equation}
 \Ex_{\Q,\N^{\co 1}}(t)\coloneqq \sum_{\y\in\frac{1}{t}\Z^d} \mult_{\Q,\N^{\co 1}}(\y) 
 = \sum_{\y\in\Z^d} \mult_{(t\cdot\Q,t\cdot\N^{\co 1})}(\y)\,,
\end{equation}
where
\begin{equation}
\begin{split}
  \mult_{\Q,\N^{\co 1}}(\y) &\coloneqq
  \begin{cases}
    \ \# \left(\text{closed full-dimensional normal cones in $\N$ containing }\y\right)\  \text{ if }\y\in\Q\,,\\
    \ 0\, \text{ else.}
  \end{cases}
\end{split}
\end{equation}
See \cref{abb:EcR} for an illustration.
\begin{lem}\label{lem:EcR}
For a polytope $\Q\subset \R^d$ and a complete fan $\N$ in $\R^d$,
\begin{equation}
\Ex_{\Q,\N^{\co 1}}(t)= \sum_{i=1}^k \ehr_{ \overline{R}_i}(t)\, , \label{eq:EcR}
\end{equation}
where $\overline{R}_i$ are the topological closures of the regions $R_i$ of the pruned inside-out polytope ${\Q\setminus \left(\bigcup \N^{\co 1}\right)}$.
\end{lem}
 \begin{proof}
 The right hand side of the equation counts lattice points in the interior of the regions $tR_i$ precisely once and lattice points in the boundaries of the regions once for every closed region the lattice point is contained in. The closed regions are the intersections of the polytope $\Q$ with the closed full-dimensional cones in $\N$. Hence every lattice $y$ point in $t\Q$ is counted with multiplicity $\mult_{(t\cdot\Q,t\cdot\N^{\co 1})}(\y)$.
 \end{proof}
\begin{figure}
\centering
\begin{subfigure}[b]{.32\textwidth}
  \centering
  \includegraphics[]{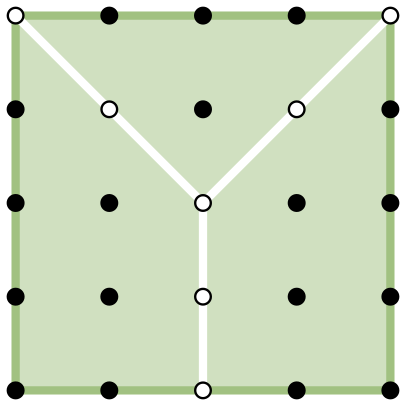}
\caption{$\Oc _{[-1,1]^2,\N^{\co 1}}(2)=18$}\label{abb:InQ}
\end{subfigure}%
\begin{subfigure}[b]{.32\textwidth}
  \centering
  \includegraphics[]{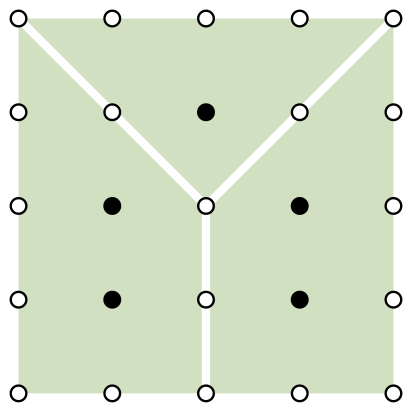}
\caption{$\Oc _{(-1,1)^2,\N^{\co 1}}(2)=5$}\label{abb:InQo}
\end{subfigure}%
\begin{subfigure}[b]{.32\textwidth}
  \centering
  \includegraphics[]{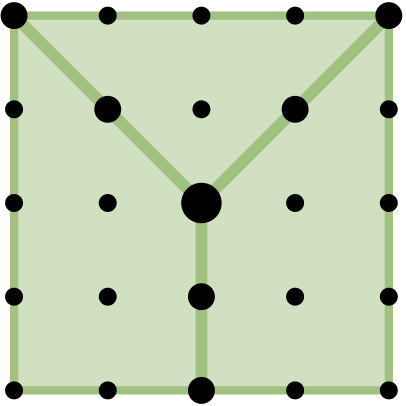}
\caption{$\Ex_{[-1,1]^2,\N^{\co 1}}(2)=33$ }\label{abb:EcR}
\end{subfigure}
\caption{Inner and cumulative pruned Ehrhart functions of the pruned inside-out polytope $[-1,1]^2\setminus\bigcup\M$ and the open pruned inside-out polytope $(-1,1)^2\setminus\bigcup\M$ illustrated. White dots are not counted, black dots are counted according to their size with multiplicity one, two, or three. 
The corresponding computations can be found in \Cref{ex:count}}\label{abb:count}
\end{figure}
 \begin{ex}\label{ex:count}
 We compute the counting functions for the pruned inside-out polytopes introduced in \Cref{ex:piop}:
 \begin{equation}
  \begin{split}
   \Oc_{(-1,1)^2,\M^{\co1}}(t)= (t^2-2t+1)+2(\tfrac{3}{2}t^2-\tfrac 5 2 t+1)=4t^2-7t+3\\
   \Ex_{[-1,1]^2,\M^{\co1}}(t)= (t^2+2t+1)+2(\tfrac{3}{2}t^2+\tfrac 5 2 t+1)=4t^2+7t+3\,.
  \end{split}
 \end{equation}
 See \Cref{abb:count} for illustrations.
 \end{ex}
\begin{thm}
\label{thm:pioehr} 
Let $\Q\setminus \left(\bigcup \N^{\co 1}\right) \subset\R^d$ be a rational pruned inside-out polytope.
Then the inner pruned Ehrhart function $ \Oc _{\Q^\circ,\N^{\co 1}}(t)$ and the cumulative pruned Ehrhart function $\Ex_{\Q,\N^{\co 1}}(t)$ agree with  quasipolynomials in $t$ of degree $ d$ for $t\in\Z_{>0}$ and are related by reciprocity:
 \begin{equation}
 (-1)^d \Oc _{\Q^\circ,\N^{\co 1}}(-t)= \Ex_{\Q,\N^{\co 1}}(t).
 \end{equation}
\end{thm}
\begin{proof}
We first use \cref{lem:InSum} to get
\begin{equation}
 \Oc _{\Q^\circ,\N^{\co 1}}(t)
 =\sum_{i=1}^k \ehr_{ R_i^\circ}(t)\,.
\end{equation}
For every $i=1,\dots, k$ we can apply Ehrhart's \cref{thm:ehr} to $\ehr_{ R_i^\circ}(t) $, hence the counting function $\Oc _{\Q^\circ,\N^{\co 1}}(t)$ is a sum of quasipolynomials, which is again a quasipolynomial.

For the second part of the claim 
we use Ehrhart--Macdonald reciprocity (\cref{thm:EhrMac}) and compute
\begin{equation}
 \begin{split}
  \Oc _{\Q^\circ,\N^{\co 1}}(t) \ =\ \sum_{i=1}^k \ehr_{ R_i^\circ}(t)
		\ =\  \sum_{i=1}^k (-1)^d \ehr_{\overline{R}_i}(-t)
		\ =\ (-1)^d \Ex_{\Q,\N^{\co 1}}(-t),
 \end{split}
\end{equation}
where the last equality follows from \cref{lem:EcR}.
\end{proof}
\begin{rem}\label{rem:pioehr}
 In the case that the polytope $Q$ and the complete fan intersect such that all the closed regions
$\overline{R}=\Q\cap N$ of the pruned inside-out polytope $\Q\setminus \left(\bigcup \N^{\co 1}\right)$ are integer polytopes, the counting functions $ \Oc _{\Q^\circ,\N^{\co 1}}(t)$ and $\Ex_{\Q,\N^{\co 1}}(t)$ agree with a \emph{polynomial} of degree $d$, by \cref{thm:ehr} and \cref{thm:EhrMac}. We will use this fact in the proof of \cref{cor:genP2}.
\end{rem}
\begin{rem}
One can certainly generalize this setting, e.g., to polyhedral complexes. The framework here is motivated by the applications below. 
\end{rem}

\section{Applications}\label{sec:appl}

After a short introduction to generalized permutahedra (\cref{ssec:prelimGenP}) we show how the tools from \cref{sec:prunedEhr} can be applied to derive known and unknown  reciprocity results for generalized permutahedra (\cref{ssec:AgenP}).
Reciprocity theorems for generalized permutahedra by Ardila and Aguiar (\cite[Propositions~$17.3$ and $17.4$]{AA17}, see \cref{cor:genP}) and extended by Karaboghossian (\cite[Theorem~$2.5$ and Theorem~$2.8$]{karaboghossian_combinatorial_2022}, see \cref{thm:hopfreci}),  
were developed by introducing a Hopf monoid structure on the vector species of generalized permutahedra and using their antipode formula to derive  polynomial invariants.
We give a new interpretation from a discrete-geometric perspective as integer point counting functions.
In  \cref{ssec:relation} we give an explanation on the relation between the results in this paper and prior results developed with Hopf-algebraic tools.
 Finally  we demonstrate why generalized permutahedra are such an interesting class of polytopes by translating 
 the reciprocity result for hypergraphic polytopes to  combinatorial statements about hypergraphs (\cref{Assec:hg}). 

\subsection{Preliminaries: Generalized permutahedra}\label{ssec:prelimGenP}
 \begin{figure}
\centering
\begin{subfigure}[b]{.25\textwidth}
  \centering
  \includegraphics[width=.95\linewidth]{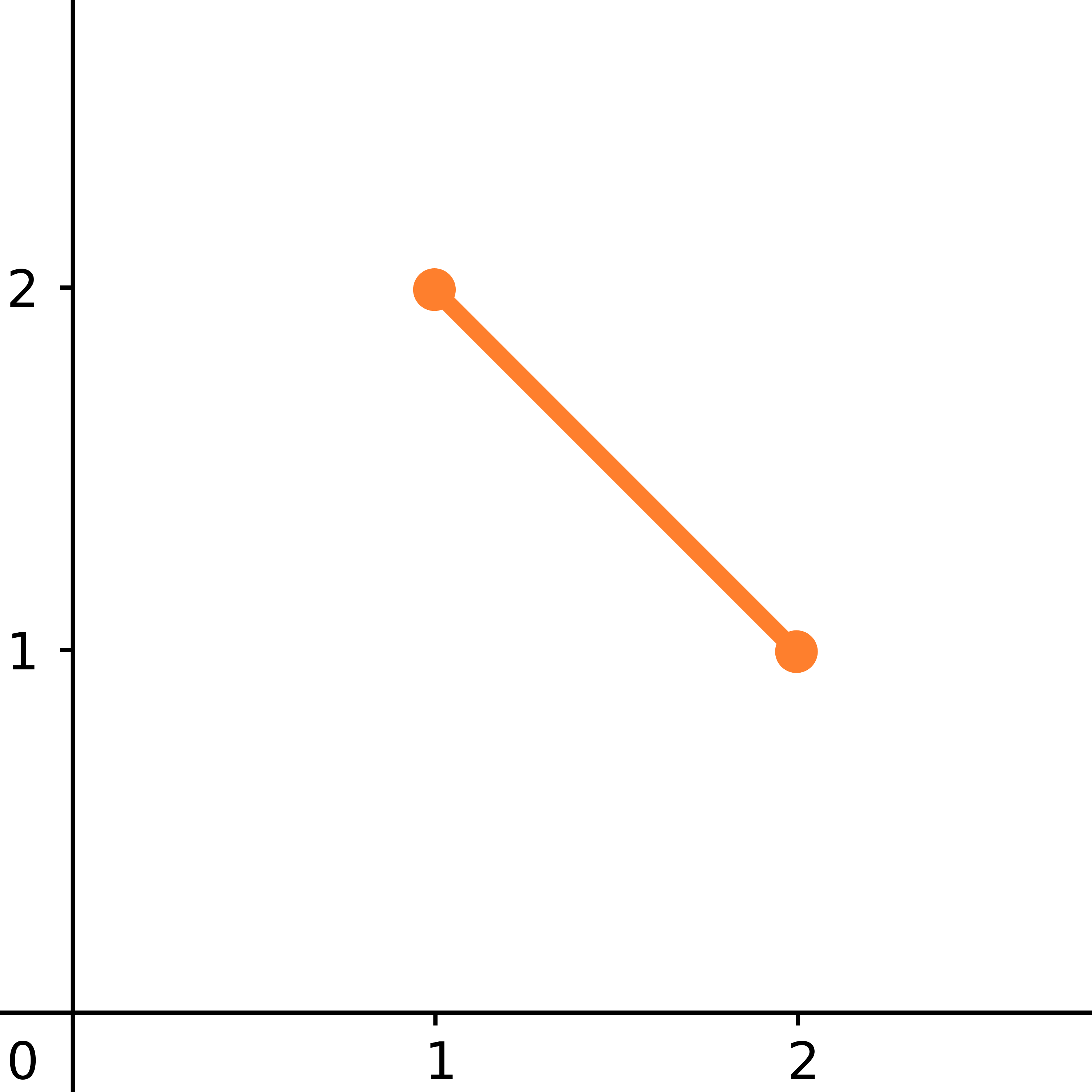}
\caption{ $\pi_{2}\subset\R^2$}
\end{subfigure}%
\begin{subfigure}[b]{.36\textwidth}
  \centering
  \includegraphics[trim=180 50 190 115, clip, width=.95\linewidth]{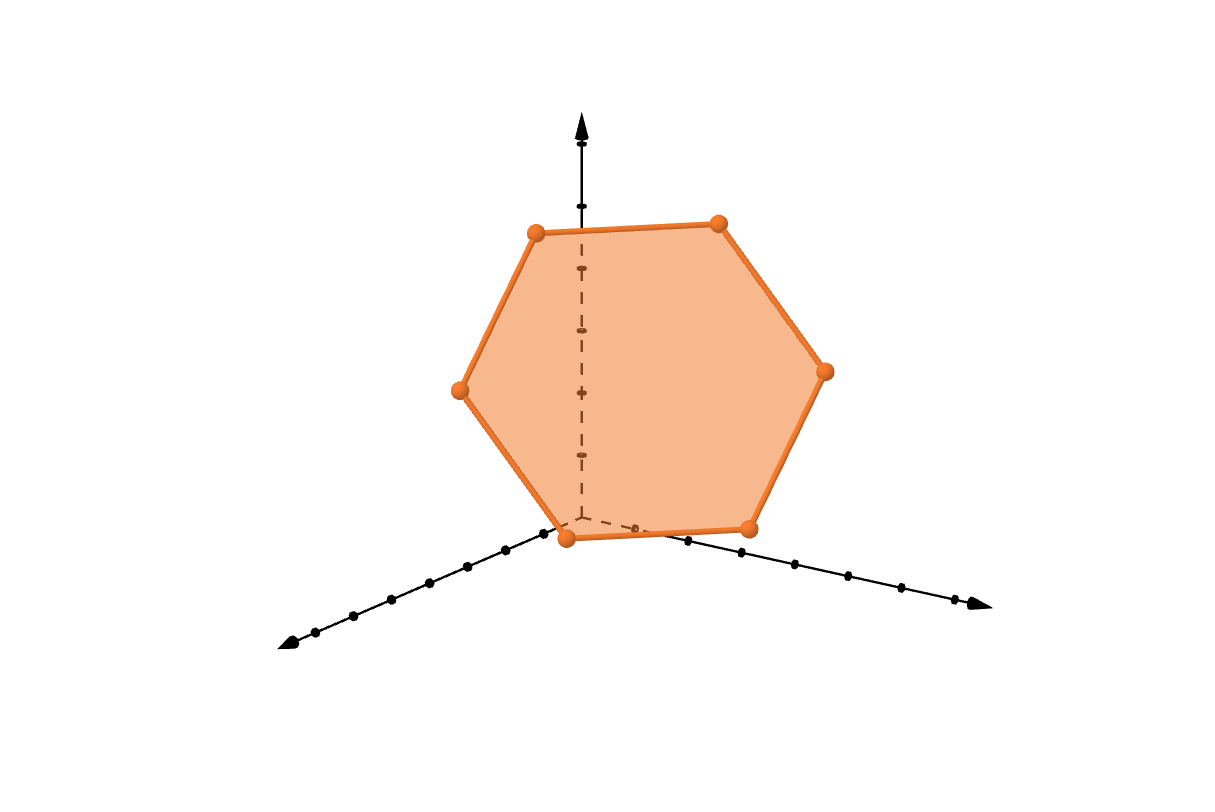}
\caption{$\pi_3\subset\R^{3}$}\label{abb:perm3}
\end{subfigure}%
\begin{subfigure}[b]{.39\textwidth}
  \centering
  \includegraphics[width=.95\linewidth]{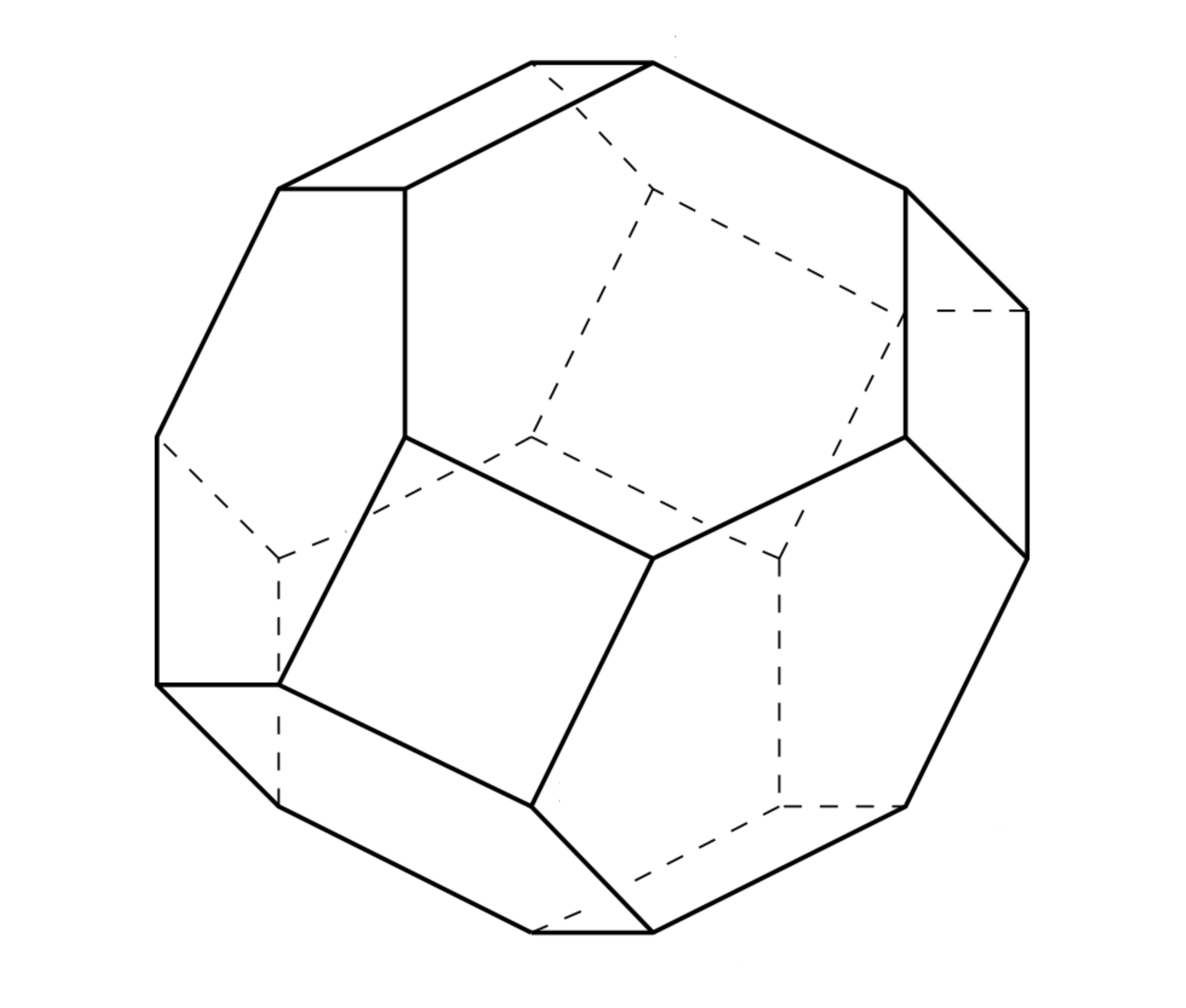}
\caption{$\pi_{4}\subset\R^4$
 \cite[p.$18$]{ziegler_lectures_1998}}
\end{subfigure}
\caption{Three standard permutahedra.
}\label{abb:perm}
\end{figure}
We define the \demph{standard permutahedron}  $\pi_d$
as the convex hull of the $d!$ permutations of the point $(1,2,\dots,d)$, that is, the standard permutahedron $\pi_d$ is defined by%
\footnote{The definition of standard permutahedron is not consistent within literature, e.g., 
Postnikov defines the standard permutahedron in a more general way: as the convex hull of all the points obtained by permuting the coordinates of an arbitrary point  \cite[Definition 2.1]{postnikov_permutohedra_2009}. 
}
 \begin{equation}
  \pi_d\coloneqq\conv\big\{(x_i)_{i\in [d]}\in \R^d \colon \{x_i\}_{i\in [d]}=[d]\big\} 
		\subset \R^d.
 \end{equation}
 \cref{abb:perm} shows some examples.  Note that the standard permutahedron is of dimension $d-1$ since all vertices are contained in a hyperplane with constant coordinate sum. 
 In our definition, standard permutahedra are integer polytopes.
 Other equivalent descriptions and references can be found in \cref{appendix}.
Let $\big(\R^d\big)^*$ be the dual vector space to $\R^d$. We identify
\begin{equation}
 \big(\R^d\big)^*=\R^{[d]}\coloneqq\{\text{maps }y\colon [d]\to\R\}
\end{equation}
and call the elements $y\in \R^{[d]}$ \demph{directions}.
Directions $y\in \R^{[d]}$ act as linear functionals on elements $x\in\R^d$ via
 \begin{equation}
  y(x)=\sum_{i=1}^d x_i y(i).
 \end{equation}
We will also exploit that primal and dual vector spaces are isomorphic.

For a direction $y\in\big(\R^d\big)^*$ we define the \demph{$y$-maximal face} $\Pol_y$ of a polytope $\Pol$ by
\begin{equation}
  \Pol_y\coloneqq\{x\in \Pol\ \colon\ y(x)\geq y(x')\quad \text{for all } x'\in \Pol\}.
 \end{equation}
 For a face $F$ of a polytope $\Pol$ define the \demph{open} and \demph{closed normal cone} 
 $N_\Pol^\circ(F)$ and $N_\Pol(F)$ to be the set of all direction that (strictly) maximize $F$ in $\Pol$, that is,
 \begin{equation}
\begin{split}
 N_\Pol(F)^\circ &\coloneqq \big\{y\in\big(\R^d\big)^* \ \colon\ \Pol_y=F\big\}\\
 N_\Pol(F) &\coloneqq \big\{y\in\big(\R^d\big)^* \ \colon\ \Pol_y\supseteq F\big\}.
\end{split}
\end{equation}
Collecting the normal cones $N_\Pol(F)$ of all faces $F$ of a polytope $\Pol$ defines the \demph{normal fan}
\begin{equation}
 \N(\Pol)\coloneqq \big\{N_\Pol(F)\ \colon\ F \text{ a face of }\Pol\big\}\,.
\end{equation}
See \cref{abb:permcon} for an example. 
Note that normal fans of polytopes form complete fans as defined in \Cref{ssec:polytopes&Ehrhart}.
 \begin{figure}
\centering
\begin{subfigure}{.47\textwidth}
  \includegraphics[trim=180 50 190 115, clip, width=\linewidth]{Permuta3}
  \caption{Standard permutahedron $\pi_{3}$}
\end{subfigure}\hfill
\begin{subfigure}{.47\textwidth}
  \includegraphics[trim=180 50 190 115, clip, width=\linewidth]{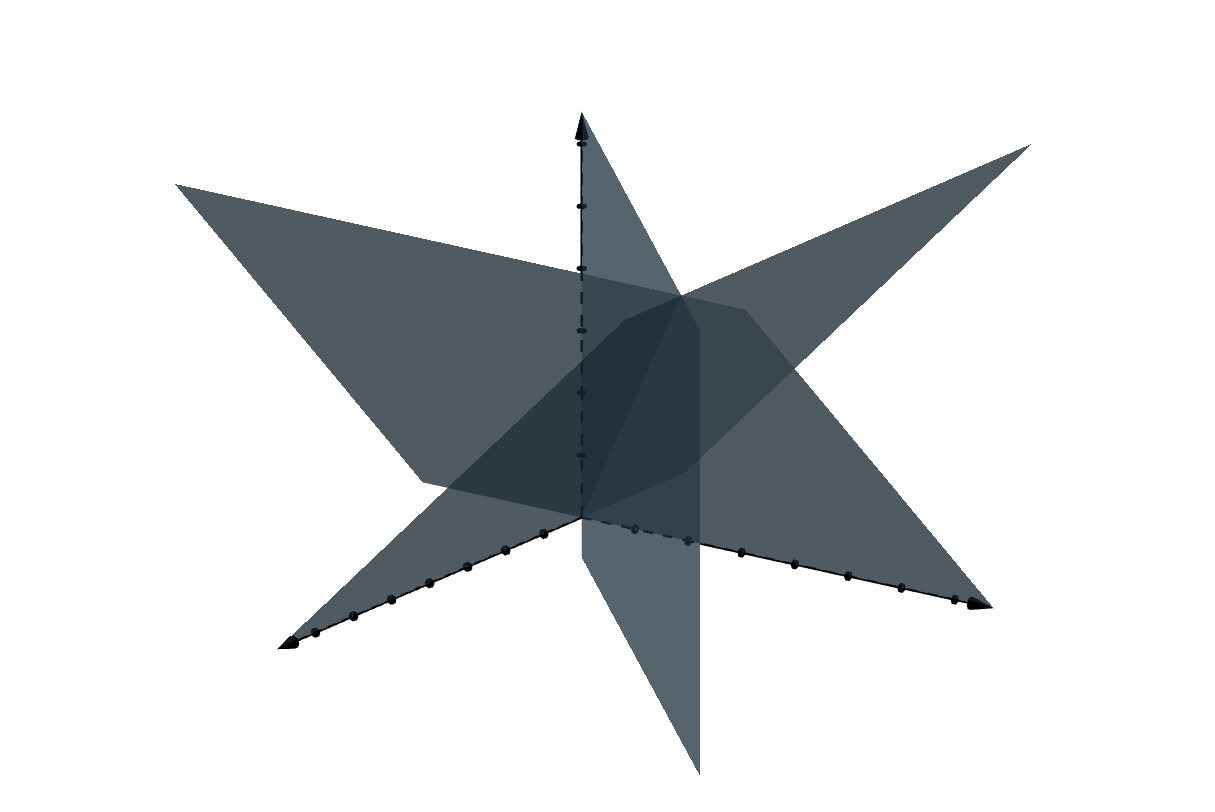}
  \caption{Braid arrangement $\B_{3}$}\label{abb:braid}
\end{subfigure}
 \caption{The standard permutahedron $\pi_{3}\subset \R^3$ (left) and the normal fan in $(\R^3)^*$ (right), where the intersection line is the normal cone $N_{\pi_{3}}(\pi_{3})$, the half hyperplanes are the normal cones of the edges and the full-dimensional cones are the normal cones of the vertices of $\pi_{3}$.
}%
\label{abb:permcon}
\end{figure}
The following is straightforward.
\begin{lem}\label{normalconerelations}
For a face $F$ of a polytope $\Pol\subset \R^d$ with dimension $\dim(F)=k$ the dimension of the normal cone is given by $\dim(\N_\Pol(F))=d-k=\codim(F)$. 
For another face $G$ of the polytope $\Pol$ we have  $F\subseteq G$ if and only if $N_\Pol(F)\supseteq N_\Pol(G)$.
\end{lem}
Recall that the codimension one fan $\N^{\co1}$ contains all cones of the complete fan $\N$ with codimension one.
This implies in particular that the codimension-one fan of the normal fan $\N^{\co1}(\Pol)$ of a polytope $\Pol$ defined in \eqref{def:codim1fan} can be described as
 \begin{equation}
 \begin{split}
   \N^{\co1}(\Pol) &= \N(\Pol)\setminus \big\{N_\Pol(v)\ \colon\ v \text{ vertex of }\Pol\big\}\\
   &= \big\{N_\Pol(F)\ \colon\ F \text{ a face of }\Pol \text{ with }\dim(F)\geq1\big\}.
 \end{split}
 \end{equation} 
The normal fan of the standard permutahedron has a nice description via the 
\demph{braid arrangement} $\B_d$, the hyperplane arrangement consisting of the finite set of hyperplanes $H_{ij}\coloneqq\{x\in\R^d\ \colon\ x_i=x_j\}$ for $i,j\in [d]$, $i\neq j$.
See \cref{abb:braid} for the example $\B_{3}$.
The connected components of $\R^d\setminus \bigcup \B_d$ are the \demph{(open) regions} of the  arrangement. 
The \demph{closed regions} of the braid arrangement are the topological closures of the open regions. They are polyhedral cones and their faces are the \demph{faces} of the braid arrangement, also called \demph{braid cones}. The braid cones can be described uniquely by compositions $[d]=T_1\uplus\dots\uplus T_k$ (\cref{lem:braid}). We therefore denote them by $\B_{T_1,\dots,T_k}$.
For more details about concepts on hyperplane arrangements see, for example, \cite{Stanley2007}.
The faces of the braid arrangement $\B_d$ form the \demph{braid fan} and the normal fan $\N(\pi_d)$  of the standard permutahedron $\pi_d$ is precisely the braid fan (see, for example, \cite[Section 4]{AA17}).

\begin{figure}
\centering
  \begin{subfigure}{.49\textwidth}
  \includegraphics[trim=180 50 190 115, clip, width=\linewidth]{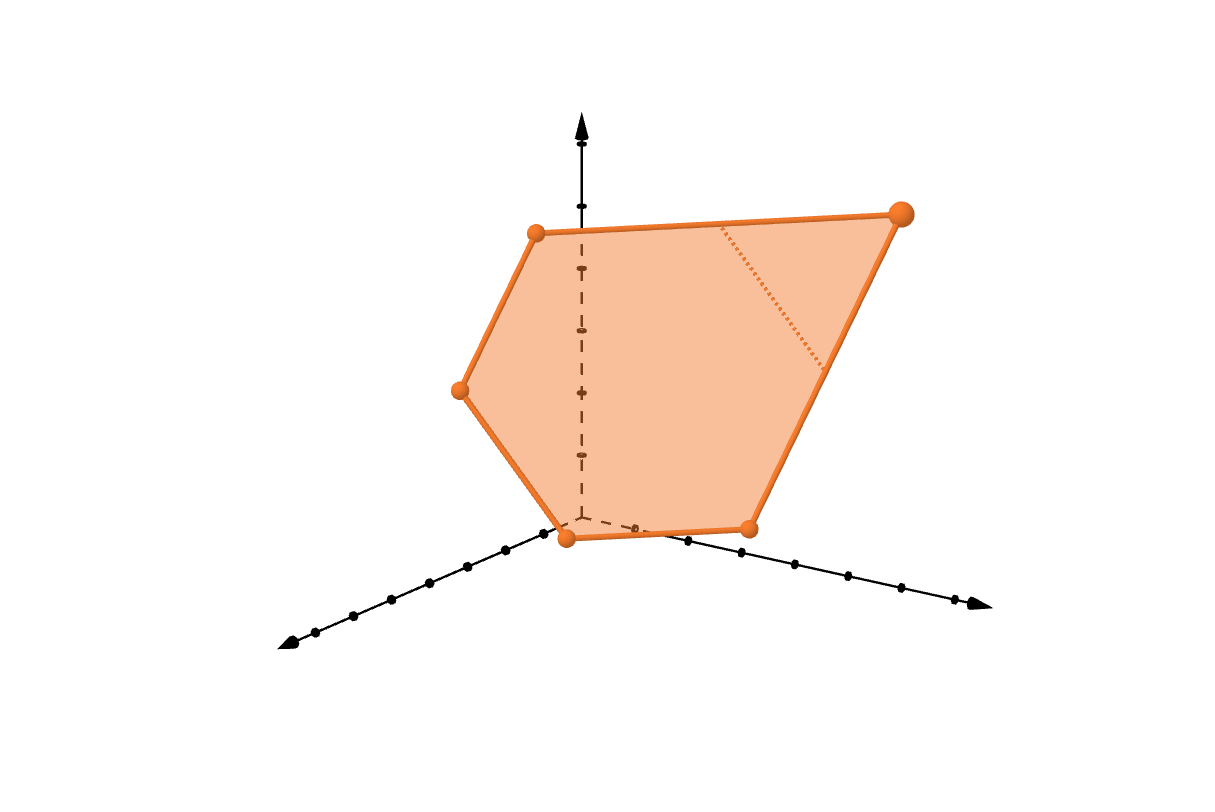}
\end{subfigure}\hfill
\begin{subfigure}{.49\textwidth}
  \includegraphics[trim=180 50 190 115, clip, width=\linewidth]{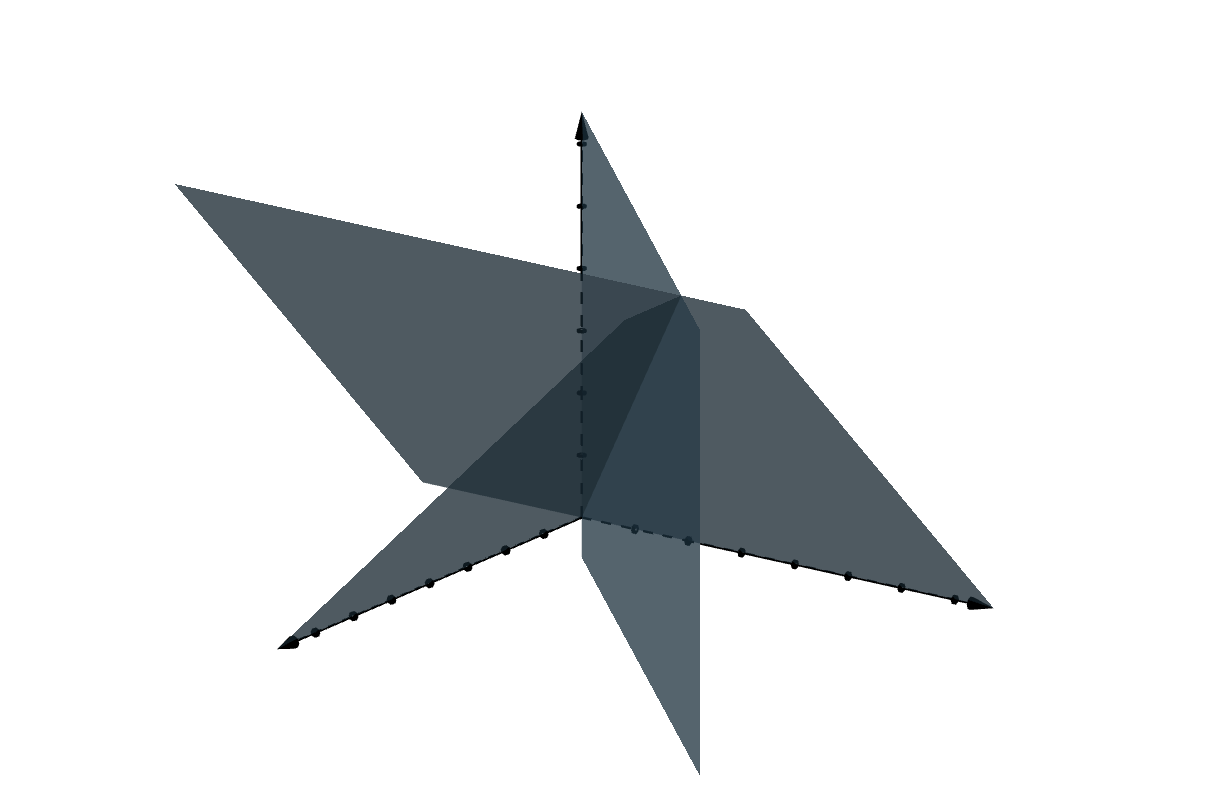}
\end{subfigure}
  \caption{
  A generalization $\Pol$ of the standard permutahedron $\pi_{3}$: here the up-right edge was moved outwards until it degenerated to a vertex. The normal cone of that ``new'' vertex is the union of the normal cones of the ``old'' degenerated edge and its adjacent vertices.
  }
  \label{fig:verallPerm}
 \end{figure}
We say a  fan $\N$ is a \demph{coarsening} of another  fan $\N'$ if every cone in $\N$ is the union of some cones in $\N'$.
A polytope $\Pol\subset \R^d$ is a \demph{generalized permutahedron} if its normal fan $\N(\Pol)$ is a coarsening of the normal fan $\N(\pi_d)$ of the standard permutahedron $\pi_d$, that is, it is a coarsening of the fan induced by the braid arrangement $\B_d$.
There are several equivalent definitions of generalized permutahedra (see, e.g., \cite{postnikov_faces_2006,castillo_deformation_2020,AA17,postnikov_permutohedra_2009} or \Cref{appendix}, where we, in particular,  provide  a self-contained proof of the hyperplane description of generalized permutahedra).

\subsection{Combinatorial reciprocity theorems for generalized permutahedra}\label{ssec:AgenP}
We restate the combinatorial reciprocity result for generalized permutahedra by \cite[Propositions~$17.3$ and $17.4$]{AA17}  in a slightly different language (see \cref{cor:genP} for the original statement) and prove it using Ehrhart theory.

\begin{thm}
\label{cor:genP2}
 Let $\Pol\subset\R^d$ be a generalized permutahedron and $m\in\Z_{>0}$. Then
 \begin{equation}
  \chi_d(\Pol)(m) \coloneqq \#\, \left(\Pol\text{-generic directions }  \y\in\big(\R^d\big)^*
			\text{ with }\y\in[m]^d \right)
 \end{equation}
 agrees with a polynomial in $m$ of degree $d$. Moreover,
 \begin{equation}
   (-1)^d \chi_d(\Pol)(-m) = \sum_{\y\in[m]^d} 
			 \# \left(\text{vertices of } \Pol_\y\right)\, .
 \end{equation}
\end{thm}

While we will extend \cref{cor:genP2} (and our proof)  in \cref{thm:moredim} and \cref{rem:beyond} below,  we provide a self-contained proof here to present a flavor of our method. 
In contrast to \cite{Billera2009, AA17, karaboghossian_combinatorial_2022} we will prove these results  without using any Hopf-algebraic method. 
Our proof gives a geometric point of view by counting integer points in  pruned inside-out cubes.
That is, we  will consider the cube 
\begin{equation}
 [1,m]^d\coloneqq\{\x\in\R^d\,\colon\, 1\leq x_i \leq m \text{ for }i=1,\dots, d \} \subset \R^d
\end{equation}
 and intersect it with the integer lattice: 
 \begin{equation}\label{eq:intgerpointsincube}
  [1,m]^d\cap \Z^d=\left\{\x\in\R^d\,\colon\,  x_i \in \{1,\dots,m\} \text{ for }i=1,\dots, d \right\}=\{1,\dots,m\}^d=[m]^d.
 \end{equation}
The same holds in the dual space $\big(\R^d\big)^*$.
Now, a direction $y\colon [d] \to [m]\in\big(\R^d\big)^*$ can be identified with an integer point $\y$ in the cube $\{1,\dots,m\}^d=[m]^d$ in the dual space. 
See \cref{abb:cubeCones}.
Before we start the proof of \Cref{cor:genP2} we need the following result.
\begin{lem}\label{lem:cube&braid}
 The intersections of the unit cube $[0,1]^d$ and the braid cones $\B_{T_1,\dots,T_k}$ for compositions $[d]=T_1\uplus\dots\uplus T_k$ are integer polytopes.
\end{lem}
\begin{proof}
 It is enough to consider the full-dimensional braid cones, since lower dimensional braid cones are faces of full-dimensional braid cones and faces of an integer polytope are integer polytopes.
 Full-dimensional braid cones $\B_{T_1,\dots,T_d}$ correspond to permutations of the coordinates, i.e., total orders on $[d]$. 
 Hence, we can think of the intersection $[0,1]\cap \B_{T_1,\dots,T_d}$ as an order polytope (see,  e.g.,  \cite[Definition~$1.1$]{Stanley1986}) of the total order on $[d]$ given by $T_1,\dots,T_d$. 
 Then Corollary~$1.3$ in \cite{Stanley1986} implies that the vertices of $[0,1]\cap \B_{T_1,\dots,T_d}$ have coordinates equal to either $0$ or $1$, hence they are integral.
\end{proof}
\begin{figure}
\centering
\begin{subfigure}[b]{.492\textwidth}
  \centering
  \includegraphics[trim=180 50 190 115, clip, width=\linewidth]{genPermuta3}
\end{subfigure}%
\begin{subfigure}[b]{.492\textwidth}
  \centering
  \includegraphics[trim=180 50 190 115, clip, width=\linewidth]{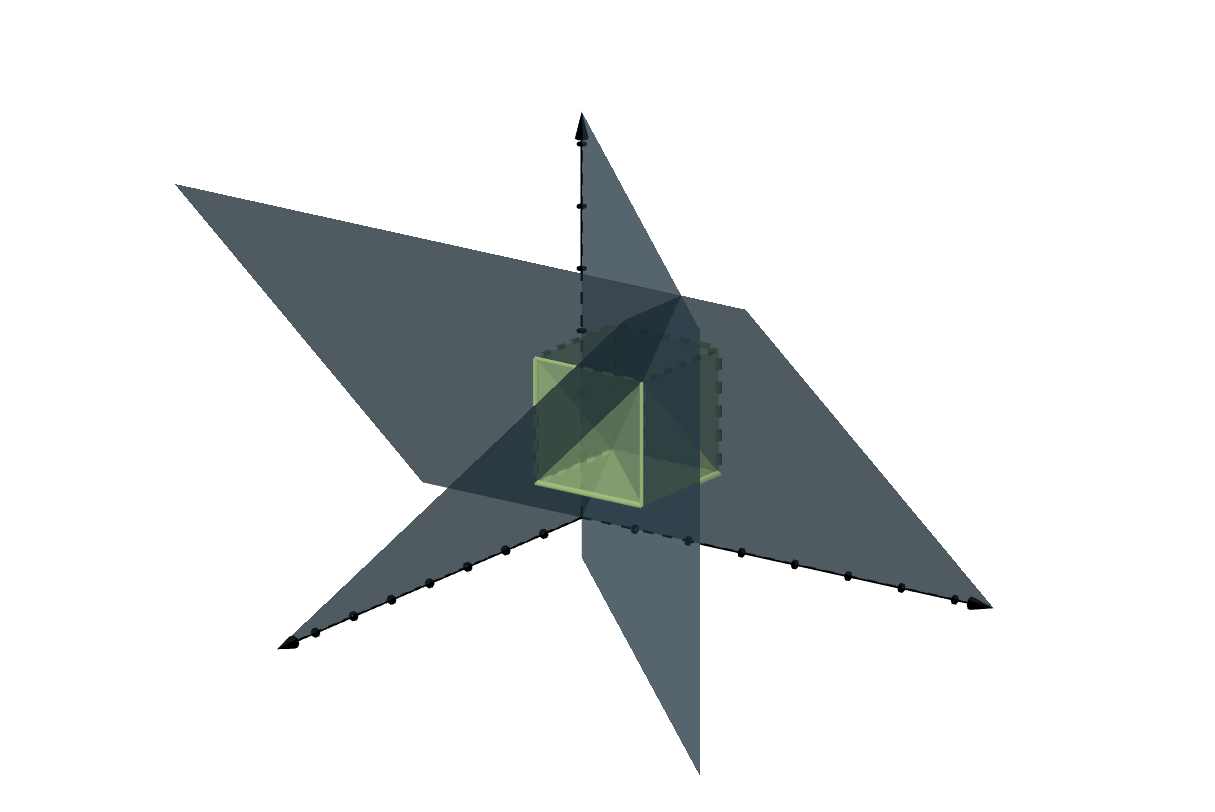}
\end{subfigure}
\caption{A generalized permutahedron in $\R^3$ (left) and its normal fan intersecting the cube $[1,2]^3$ (right).
}
\label{abb:cubeCones}
\end{figure}
\begin{proof}[Proof of \cref{cor:genP2}]
We will argue 
in the dual space $\big(\R^d\big)^*$ and its integer lattice; to simplify notation we will not always explicitly point that out.
Let us recall that $\y\in \big(\R^d\big)^*$ being $\Pol$-generic means that the $\y$-maximal face of $\Pol$ is a vertex, that is, $\y$ is contained in a full-dimensional cone of the normal fan $\N(\Pol)$. 
So the direction $\y$ is \emph{not} contained in any cone $N$ in the codimension-one fan $\N(\Pol)^{\co 1}$. Hence,
\begin{equation}
\begin{split}
 \chi_d(\Pol)(m) &=  \#\, \left(\Pol\text{-generic directions }  \y\in\big(\R^d\big)^*
			\text{ with }\y\in[m]^d \right)\\
		&= \#{\{\y\in [1,m]^d\cap\Z^d \ \colon\ \y\text{-maximum face $\Pol_\y$ is a vertex}\}}\\
		&= \#{\{\y\in [1,m]^d\cap\Z^d \ \colon\ \y\in N\in\N(\Pol) \text{ with } N \text{ full-dimensional} \}}\\
		&= \#{\{\y\in [1,m]^d\cap\Z^d \ \colon\ \y\notin N \text{ for all }N\in\N(\Pol) \text{ with codimension } \geq 1 \}}\\
		&=\#{\left(\,\left([1,m]^d\setminus \bigcup\N(\Pol)^{\co 1}\right)  \cap\Z^d  \,\right) }\\
		&= \Oc_{(0,1)^d,\N(\Pol)^{\co 1}}(m+1)\,,
\end{split}
\end{equation}
where we use in the last line that $[1,m]^d\cap \Z^d= (0,m+1)^d\cap \Z^d=(m+1)\cdot(0,1)^d\cap\Z^d$.
With \cref{lem:cube&braid} we know that the unit cube and the normal fan $\N(\Pol)$ intersect producing integer regions. 
Therefore, using \cref{thm:pioehr} and \cref{rem:pioehr}, polynomiality of $\chi_d(\Pol)(m) $ follows.
With the above equality and \cref{thm:pioehr}  at hand, we compute
\begin{equation}
\begin{split}
 (-1)^d\chi_d(\Pol)(-m) &= (-1)^d \Oc _{(0,1)^d,\N(\Pol)^{\co 1}}(-m+1)\\
		&= (-1)^d \Oc _{(0,1)^d,\N(\Pol)^{\co 1}}(-(m-1))\\
		&= \Ex_{[0,1]^d,\N(\Pol)^{\co 1}}(m-1)\\
		&= \sum_{\y\in \frac{1}{m-1}\Z^d} \mult_{[0,1]^d,\N(\Pol)^{\co1}}(\y) \\
		&= \sum_{\y\in\Z^d} \mult_{[0,m-1]^d,\N(\Pol)^{\co1}}(\y)\,.
\end{split}
\end{equation}
Every cone in the braid fan contains the line $L=\lambda (1,\dots,1)$.
Therefore, the fans $\N(\Pol)$ and $\N(\Pol)^{\co1}$ are invariant under  translations by vectors in the line $L$ and scaling.
So we can shift the cube $[0,m-1]^d$ to $[1,m]^d$ and this bijection not only preserves the number of integer points but also their multiplicities $\mult_{[1,m]^d,\N(\Pol)^{\co1}}$ with respect to the fan $\N(\Pol)$. Hence,
\begin{equation}
\begin{split}
 (-1)^d\chi_d(\Pol)(-m)
		&= \sum_{\y\in\Z^d} \mult_{[1,m]^d,\N(\Pol)^{\co1}}(\y) \\
		&=\sum_{\y\in[1,m]^d\cap\Z^d} 
					\# \left(\text{closed full-dimensional normal cones that contain }\y\right)\\
		&=\sum_{\y\in[m]^d} 
					\# \left(\text{closed normal cones of vertices that contain }\y\right)\\
		&=\sum_{\y\in[m]^d} \# \left(\text{vertices of } \Pol_\y\right)\,,
\end{split}
\end{equation}
where we make use of \Cref{normalconerelations}.
\end{proof}

We can extend \cref{cor:genP2} above to faces of arbitrary dimension.
\begin{thm}\label{thm:moredim}
 For a generalized permutahedron $\Pol\subset \R^d$ and $k=0,\dots,d-1$,
\begin{equation}
\begin{split}
 \chi_{d,k}(\Pol)(m) \coloneqq& \ \#{\{\y\in [m]^d \ \colon\ \y\text{-maximum face $\Pol_\y$ is a $k$-face}\}}
\end{split}
\end{equation}
 agrees with a polynomial of degree $d-k$, and
\begin{equation}
\begin{split}
 (-1)^{d-k}\chi_{d,k}(\Pol)(-m) &= \sum_{\y\in[m]^d} 
			 \# \left(k\text{-faces of } \Pol_\y\right).
\end{split}
\end{equation}
\end{thm}
Before we prove the theorem
we extend the notion of codimension-one fans to arbitrary dimensions by defining the \demph{codimension-$\mathbf{k}$ fan} $\N^{\co k}$ as
\begin{equation}
\begin{split}
	\N^{\co k}\coloneqq& 
	\{N\in\N(\Pol)\,\colon\, \codim(N)\geq k \},
\end{split}
\end{equation}
that is, for a polytope $\Pol$, 
\begin{equation}
 \N(\Pol)^{\co k} = \{N_\Pol(F)\ \colon\ F \text{ a face of }\Pol \text{ with }\dim(F)\geq k\}\,.
\end{equation}
For a polytope $\Q\subset \R^d$ and $k\geq 0$ we define the \demph{$k$-pruned inside-out polytope} as
\begin{equation}
\begin{split}
  \left(\Q\cap \bigcup \N^{\co k}  \right) \setminus \left(\bigcup \N^{\co k+1}\right)
  =\Q\cap\biguplus \left\{N^\circ\,\colon\, N\in\N^{\co k} \right\} \,.
\end{split}
\end{equation}
Note this is consistent with the notation in the beginning of this section.
As before, for a polytope $\Q\subset \R^d$ the open $k$-pruned inside-out polytope $\left(\Q^\circ\cap \bigcup \N^{\co k}  \right) \setminus \left(\bigcup \N^{\co k+1}\right) $ is the disjoint union of relatively open $(d-k)$-dimensional polytopes, 
namely, the intersection of $\Q^\circ$ with the relatively open cones in $\N$ of codimension $k$.
\begin{proof}[Proof of \Cref{thm:moredim}.]
We compute
\begin{equation}
\begin{split}
 \chi_{d,k}(\Pol)(m) &= \ \#{\big\{\y\in [1,m]^d\cap\Z^d \ \colon\ \y\text{-maximum face $\Pol_\y$ is a $k$-face}\big\}}\\
		&=\#{\Big(\Big(\Z^d\cap(0,m+1)^d\cap \bigcup\N(\Pol)^{\co k} \Big) \setminus 
						\Big(\bigcup\N(\Pol)^{\co k+1} \Big)\Big)}\\
		&=\#\Big(\Big(\biguplus_{\substack{N\in\N(\Pol)\\ \dim N=d-k}} N^\circ \cap (0,m+1)^d\Big) \cap \Z^d  \Big)\,.
% 		&= \sum_{\substack{N\in\N(\Pol)\\ \dim N=d-k}} \ehr_{N^\circ\cap(0,1)}(m+1)\,,
\end{split}
\end{equation}
The intersection $N^\circ \cap (0,m+1)^d$ is the relative interior of a polytope. 
Moreover, since $N^\circ$ is an open cone containing the origin,
$N^\circ \cap (0,m+1)^d$ is the $(m+1)$\textsuperscript{st} dilate of $N^\circ \cap (0,1)^d$.
Hence,
\begin{equation}\label{eq:sumofehrkdim}
 \chi_{d,k}(\Pol)(m)
 =\#\Big(\Big(\biguplus_{\substack{N\in\N(\Pol)\\ \dim N=d-k}} N^\circ \cap (0,m+1)^d\Big) \cap \Z^d  \Big)
 = \sum_{\substack{N\in\N(\Pol)\\ \dim N=d-k}} \ehr_{N^\circ\cap(0,1)^d}(m+1)\,.
\end{equation}
Using again \cref{lem:cube&braid} and Ehrhart's \cref{thm:ehr} we obtain polynomiality for $ \chi_{d,k}(\Pol)(m)$.

With Ehrhart--Macdonald reciprocity (\cref{thm:EhrMac}) we  compute
\begin{align}
 (-1)^{d-k}\chi_{d,k}(\Pol)(-m)&=(-1)^{d-k} \sum_{\substack{N\in\N(\Pol)\\ \dim N=d-k}} \ehr_{N^\circ\cap(0,1)^d}(-m+1)\\
	&=\sum_{\substack{N\in\N(\Pol)\\ \dim N=d-k}} (-1)^{d-k}  \ehr_{N^\circ\cap(0,1)^d}(-(m-1))\\
	&=\sum_{\substack{N\in\N(\Pol)\\ \dim N=d-k}} \ehr_{N\cap[0,1]^d}(m-1)\label{eq:kclosederhhartsum}\\ 
	&=\sum_{\substack{N\in\N(\Pol)\\ \dim N=d-k}} \#\big(N\cap[0,m-1]^d\cap \Z^d\big)\,.
\end{align}
Here we use, as in the proof of \cref{cor:genP2}, that the normal fan of a generalized permutahedron is a coarsened braid fan and therefore is invariant under scaling and shifts by $\lambda (1,\dots,1)$ for  $\lambda\in \R$.
So,
\begin{equation}
\begin{split}
 (-1)^{d-k}\chi_{d,k}(\Pol)(-m)
	&=\sum_{\substack{N\in\N(\Pol)\\ \dim N=d-k}} \#\left(N\cap[1,m]^d\cap \Z^d\right)\\
	&=\sum_{\substack{\y\in[1,m]^d\cap \Z^d}} \#\left( (d-k)\text{-dimensional cones }N\in\N(\Pol) \text{ that contain }\y  \right)\\
	&= \sum_{\y\in[m]^d} \# \left(k\text{-faces of } \Pol_\y\right)\,,
\end{split}
\end{equation}
 applying \cref{normalconerelations} in the last equality.
\end{proof}
\begin{rem}\label{rem:beyond}
At the heart of the proofs of \Cref{cor:genP2} and \Cref{thm:moredim} lie  sums of Ehrhart polynomials and  the reciprocity results are applications of Ehrhart-Macdonald reciprocity (\cref{thm:EhrMac}):
Recall \eqref{eq:sumofehrkdim} and \eqref{eq:kclosederhhartsum} from the proof of \cref{thm:moredim}.
One can see that for a generalized permutahedron $\Pol$ any combination of Ehrhart polynomials as in \eqref{eq:sumofehrkdim} and \eqref{eq:kclosederhhartsum} results in a polynomial counting function
\begin{equation}\label{eq:weighted}
\begin{split}
  \chi_{d,\boldsymbol{\alpha}}(\Pol)(m)
  \coloneqq& \sum_{\substack{N\in\N(\Pol)}} \alpha_N \ehr_{N^\circ\cap(0,1)^d}(m+1)\\
     =& \sum_{\substack{F \text{ a face of }\Pol}} \alpha_{N_\Pol(F)} \ehr_{N_\Pol(F)^\circ\cap(0,1)^d}(m+1)
 \end{split}
\end{equation}
for coefficients $\alpha_N$. This provides a combinatorial reciprocity result
\begin{equation}
\begin{split}
 \chi_{d,\boldsymbol{\alpha}}(\Pol)(-m)&=\sum_{\substack{N\in\N(\Pol)}} (-1)^{\dim N} \alpha_N \ehr_{N\cap[0,1]^d}(m-1)\\
 &=\sum_{\substack{F \text{ a face of }\Pol}} (-1)^{d-\dim F} \alpha_F \ehr_{N_\Pol(F)\cap[0,1]^d}(m-1) \,.
 \end{split}
\end{equation}
 \Cref{thm:moredim} (and therefore also \Cref{cor:genP2}) is a reformulation of this general result with coefficients
\begin{equation}
 \alpha_N=\begin{cases}
           1 &\text{ if }\dim N =d-k\\
           0 &\text{ else}
          \end{cases}
          \qquad\text{ for } k=0,1,\dots,d-1\,.
\end{equation}
\end{rem}
\begin{rem}\label{rem:typeB}
 We observe that we used the following properties of generalized permutahedra in the proofs of \cref{cor:genP2} and \cref{thm:moredim} 
\begin{enumerate}[nolistsep]
 \item the intersection of the unit cube and the normal fan of a generalized permutahedron form integer pruned inside-out polytopes,\label{integral}
 \item every cone $N$ in the normal fan $\N(\Pol)$ of a generalized permutahedron $\Pol$ contains the line  $L=\left\{\lambda(1,\dots,1)\,\colon\, \lambda\in\R \right\}$.\label{lineality}
\end{enumerate}
The first property \ref{integral} can be weakened to rational intersections leading to a quasipolynomiality result.
Considering normal fans without property \ref{lineality} produces  similar but more complicated statements, since the shift of the cube $[0,m-1]^d$ to the cube $[1,m]^d$ can not be performed in general.
Nevertheless, the framework of pruned inside-out polytopes can be applied to generate reciprocity results for generalized permutahedra in other types (see, e.g., \cite{ardila_coxeter_2020}).
This will be explored in a future paper.
\end{rem}

\subsection{Relation to polynomial invariants from  Hopf monoids}\label{ssec:relation}
In this section we compare our results to the polynomial invariants from Hopf monoids developed in \cite{AA17, karaboghossian_combinatorial_2022}. 
This paper was motivated by giving a geometric interpretation of the combinatorial reciprocity theorems in \cite{AA17}.

\label{ssec:notation}%take label with notation paragraphs for reference in appendix
In the Hopf--algebraic setting it is convenient to work with vector spaces with unordered base. We briefly introduce the notation, which we also use in \cref{Assec:hg}.
For a non-empty finite set $I$ let $\R I $ be the real vector space with distinguished, unordered basis $I$.
The elements $i\in I$ with are denoted $b_i$ when we want to distinguish the elements $i$ in the set $I$ from the corresponding basis vector $b_i$ in the vector space $\R I$.
Moreover, we identify an element $\sum_{i\in I}x_i b_i$ in the vector space $\R I$ with the tuple $(x_i)_{i\in I}$ for  $x_i\in \R$.
For the disjoint union $I=S\uplus T$ of two finite sets $S, T$ the equality $\R S\times \R T=\R I=\R T\times \R S$ holds, which is handy in combinatorial contexts.
Similarly, the dual vector space $(\R I)^*$ can be  interpreted as
\begin{equation}
 (\R I)^*=\R^I\coloneqq\{\text{maps }y\colon I\to\R\}\,.
\end{equation}
Recall that the elements $y\in \R^I$ are called directions.
They act as linear functionals on elements $x=\sum_{i\in I} x_ib_i\in\R I$ via
 \begin{equation}
  y\Big(\sum_{i\in I} x_ib_i\Big)=\sum_{i\in I} x_i y(i).
 \end{equation}
 For a finite set $I$ with $\lvert I\rvert \eqqcolon d$ we can identify $\R^d \cong\R I$ and  $\R^I\simeq \big(\R^d\big)^*$ by fixing a bijection $\sigma\colon I\to [d]\coloneqq\{1,\dots,d\}$. Via this bijection we may also assume $I=[d]$.
 In the context of this paper those two notations can be used interchangeably. 

An introduction to the theory of Hopf monoids can be found in, e.g., \cite{AA17}, \cite{aguiar_monoidal_2010} and is omitted here.
For a Hopf monoid on the ground set $I$, a character $\zeta$, and an element $x$ in the Hopf monoid, there is a polynomial invariant 
\begin{equation}
 \chi_I^\zeta(x)(m)\coloneqq \sum_{I=S_1\sqcup\dots\sqcup S_m} \left(\zeta_{S_1}\otimes\dots\otimes\zeta_{S_m}\right)\circ\Delta_{S_1,\dots,S_m}(x) \,,
\end{equation}
where the sum is over all compositions and $\Delta$ denotes the coproduct of the Hopf monoid.
Using the antipode $s_I$ of the Hopf monoid one obtains the reciprocity relation
\begin{equation}
 \chi_I^\zeta(x)(-m)=\chi_I\left(s_I\left(x\right)\right)\left(m\right)
\end{equation}
which gives an interpretation for negative integers \cite[Section 16]{AA17}.
In \cite{AA17} Aguiar and Ardila define a Hopf monoid structure on the species of generalized permutahedra and then obtain combinatorial formulas for the polynomial invariant $\chi_I(x)(m)$ and $\chi_I(x)(-m)$ for $m\in\Z_{>0}$ using the basic character, which takes values in $\{0,1\}$.
\begin{thm}[{\cite[Propositions~$17.3$ and $17.4$]{AA17}}]\label{cor:genP}
 At a positive integer $m\in\Z_{>0}$ the basic polynomial invariant $\chi$ of a generalized permutahedron $\Pol\subset \R I$ is given by
 \begin{equation}
  \chi_I(\Pol)(m) = \#\, \left(\Pol\text{-generic directions }  y\colon I \to [m]\right)
 \end{equation}
and 
 \begin{equation}
   (-1)^{\lvert I\rvert} \chi_I(\Pol)(-m) = \sum_{y\colon I \to [m]} \# \left(\text{vertices of } \Pol_y\right)\, .
 \end{equation}
\end{thm}
This result was obtained earlier but stated differently by Billera, Jia, and Reiner  using a similar Hopf-algebraic approach (using the antipode) on quasisymmetric functions and matroids \cite[Theorem~9.2.~(v)]{Billera2009}.
We have seen in \Cref{ssec:AgenP} how this result can be understood using pruned inside-out cubes.
Recently, \cref{cor:genP} was generalized in 
\cite{karaboghossian_combinatorial_2022}.
\begin{thm}[{\cite[Theorem~$2.5$ and Theorem~$2.8$]{karaboghossian_combinatorial_2022}}]\label{thm:hopfreci}
 Let $\zeta$ be a character of the Hopf monoid of generalized permutahedra $GP$, $I$ a finite set and $\Pol\in GP[I]$ a generalized permutahedron. Then,
 \begin{equation}
  \chi_I^\zeta(\Pol)(m)=\sum_{F\text{ a face of }\Pol} \zeta(F)\lvert N_\Pol^\circ(F)_m\rvert\,,
 \end{equation}
 and
  \begin{equation}
  \chi_I^\zeta(\Pol)(-m)=\sum_{F\text{ a face of }\Pol} (-1)^{\lvert I\rvert-\dim F} \zeta(F)\lvert N_\Pol(F)_m\rvert\,,
 \end{equation}
\end{thm}
Here elements in the sets $\N_\Pol^\circ(F)_m= [m]^I\cap \N_\Pol^\circ(F)$ and $\N_\Pol(F)_m= [m]^I\cap\N_\Pol(F)$ are called the colorings $c\colon I\to[m]$ that are \demph{strictly compatible}, respective \demph{compatible} with $F$.
Those can easily be understood as the integer points in the open normal cone of the face $F$ intersected with the $(m+1)$\textsuperscript{st} dilate of the open unit cube $(0,1)^I$, so
$\lvert N_\Pol^\circ(F)_m\rvert$ agrees with the Ehrhart polynomial $\ehr_{N_\Pol(F)^\circ\cap(0,1)^d}(m+1)$.
Similarly, the set $N_\Pol(F)_m$ can be recognized as the integer points in the closed  normal cone of the face $F$ intersected with the closed cube $[1,m]^I$. 
After shifting the cube as in the proof of \Cref{thm:moredim}, we can see that $\lvert N_\Pol(F)_m\rvert$ agrees with the Ehrhart polynomial of $N_\Pol(F)\cap [0,1]^I$.
Hence, the polynomial invariants can be interpreted as sums of Ehrhart polynomials that are weighted by the character $\zeta(F)$, compare \cref{rem:beyond}.

With a view towards applications our  Ehrhart-theoretic approach has some advantages.
One strength is that the weights in \eqref{eq:weighted} can be chosen arbitrarily, while in the Hopf-theoretic setting the character needs to fulfill certain axioms.
This, for example, does not allow to interpret the combinatorial reciprocity result in \Cref{thm:moredim} as an instance of \Cref{thm:hopfreci}. A character taking value one on $k$-dimensional faces and zero elsewhere would not fulfill compatibility with multiplication in the Hopf monoid of generalized permutahedra.
Another advantage will be the extension to  generalized permutahedra in other types, as mentioned before (\Cref{rem:typeB}).
This seems to be  very hard from the Hopf monoid setting (see, e.g., \cite[Theorem~$6.1$]{AA17}, \cite[Section~$9$]{ardila_coxeter_2020}).

\subsection{Hypergraphs and their polytopes}\label{Assec:hg}
Generalized permutahedra are an especially interesting class of polytopes, due to their many interesting combinatorial subclasses such as graphical zonotopes, matroid polytopes, hypergraphic polytopes, and many more.
In this section we illustrate this fruitful connection between combinatorics and geometry 
proving a combinatorial reciprocity result for hypergraphs,
which generalizes Stanley's famous theorem about the chromatic polynomial for graphs.
Aval, Karaboghossian, and Tanasa use a  Hopf-theoretic ansatz similar to that of Ardila and Aguiar to derive the reciprocity theorem for hypergraph colorings \cite{Aval2020}. 
They define a basic polynomial invariant on hypergraphs and give combinatorial interpretations. 
A general version of this can be found in \cite{karaboghossian_combinatorial_2022}. 
For convenience we demonstrate the technique for a special case of orientation, that we call heading.
We give another perspective and proof by applying \cref{cor:genP} (reciprocity for generalized permutahedra) and exploiting geometric and combinatorial properties of the hypergraph and its associated polytope.
This approach is also described  as alternative proof for the general case in \cite{karaboghossian_combinatorial_2022}%
\footnote{There also seems to be a polytopal approach by Alexander Postnikov, mentioned in \cite[Acknowledgments]{Aval2020} and on \url{http://math.mit.edu/~apost/courses/18.218_2016/} (Lecture 19. W 03/16/2016), but to the best of our knowledge no reference is available.}.
 A \demph{hypergraph} $h=(I,E)$ is a pair of a finite set  $I$  of \demph{nodes}%
 \footnote{We decided to use the less common term \emph{nodes} for hypergraphs to distinguish them from the \emph{vertices} of a polytope.}
 and a finite multiset $E$ of non-empty subsets $e\subseteq I$  called \demph{hyperedges}.
Note that we allow multiple edges and edges consisting of only one node.
For simplicity we will often assume without loss of generality that the node set $I$ equals $\{1,\dots,d\}=[d]$ for $d=\lvert I\rvert$, since all the claims in this section are invariant under relabeling the set $I$. 
In a similar fashion we might switch back and forth between the two vector space notations $\R I\simeq \R^d$ and $\R^I\simeq\big(\R^d\big)^*$ (see \cref{ssec:notation}). 

 For every hypergraph $h$ we define the corresponding \demph{hypergraphic polytope} $\Pol(h)\subset \R I$  as the following Minkowski sum of simplices:
\begin{equation}
 \Pol(h)=\sum_{e\in E}\Delta_{e} \ \subset \R I
\end{equation}
where
 \begin{equation}
  \Delta_{e}= \conv\{b_i\ \colon\ i\in e\},\quad 
  \text{for a hyperedge }e\subseteq I\, 
 \end{equation}
 and $b_i$ are the basis vectors for $i\in I$.
An example is depicted in \cref{abb:hgPolyMink}. Hypergraphic polytopes have been studied (sometimes as Minkowski sum of simplices) in, e.g., \cite{agnarsson_special_2017,Benedetti2019}.
Hypergraphs are in bijection with hypergraphic polytopes  and they form a subclass of generalized permutahedra (see \cref{appendix} or, e.g., \cite[Proposition~6.3.]{postnikov_permutohedra_2009}).

\begin{figure}
 \includegraphics[width=\textwidth]{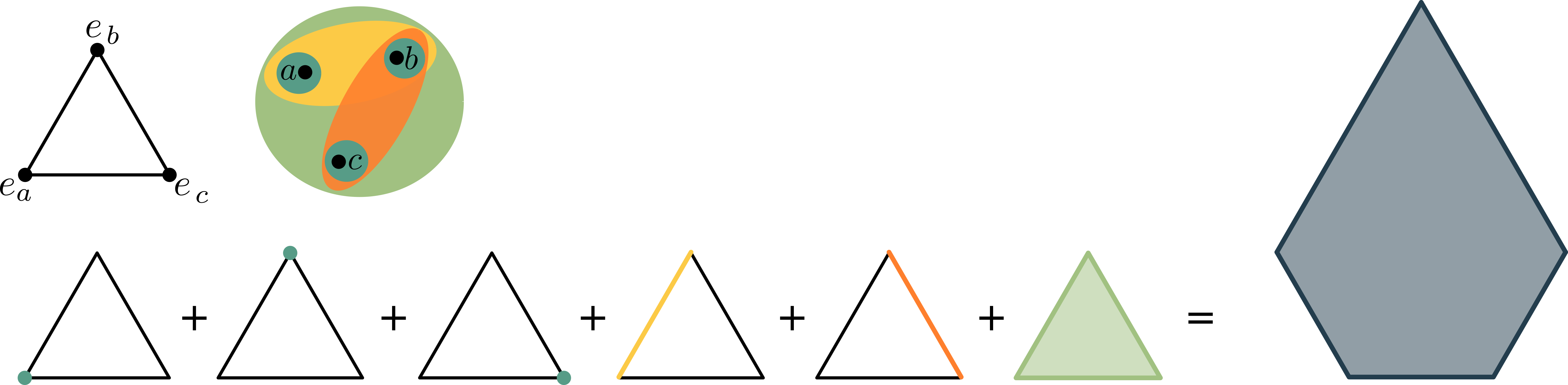}
 \caption{The hypergraph $h=\left(\left\{a,b,c\right\},\left\{\left\{a,b,c\right\},\left\{a,b\right\},\left\{b,c\right\},\left\{a\right\},\left\{b\right\},\left\{c\right\}\right\}\right)$ and its hypergraphic polytope $\Pol(h)$.}\label{abb:hgPolyMink}
\end{figure}

The vertices of graphic polytopes are described by the acyclic orientations of the corresponding graph \cite[Corollary 4.2]{Zaslavsky1991}. 
We will give an analogous statement and proof for hypergraphic polytopes. 
In order to do so we need the subsequent definitions following%
\footnote{Some of the definitions are also mentioned by Postnikov (\url{http://math.mit.edu/~apost/courses/18.218_2016/} Problem set 2, Problem 6).}
\cite{Aval2020}. 
 A \demph{heading}%
 \footnote{We have chosen to call this generalization of orientations \emph{heading} to distinguish it from other definitions of orientations for hypergraphs.}%
 $\Or$ of a hypergraph $h=(I,E)$ is a map $\Or\colon E\to I$ 
 such that for every hyperedge $e\in E$ we have $\Or(e)\in e$. 
 In other words the heading  $\Or$ picks for every hyperedge $e$ a node $i=\Or(e)\in e$ within that hyperedge. 
 We will call that node $\Or(e)$ the \demph{head of the hyperedge} $e$.
 An \demph{oriented cycle} in a heading $\Or$ of a hypergraph $h$ is a sequence $e_1,\dots, e_\ell$ of
 hyperedges such that
 \begin{equation}
\begin{split}
  &\Or(e_1)\in e_2\setminus \Or(e_2)\\
  &\Or(e_2)\in e_3\setminus \Or(e_3)\\
   &\qquad\quad\vdots\\ 
   &\Or(e_{\ell-1})\in e_{\ell-1}\setminus \Or(e_{\ell-1})\\
   &\Or(e_\ell)\in e_1\setminus \Or(e_1).
\end{split}
\end{equation}
A heading $\Or$ of a hypergraph $h$ is called \demph{acyclic} if it does not contain any oriented cycle. 
See \cref{abb:hg_def} for some examples.
Note that the notions of heading and acyclic here are  special cases of the notions in \cite{Benedetti2019, karaboghossian_combinatorial_2022, Reff2012, Rusnak2013}.
\begin{figure}
 \includegraphics[width=\textwidth]{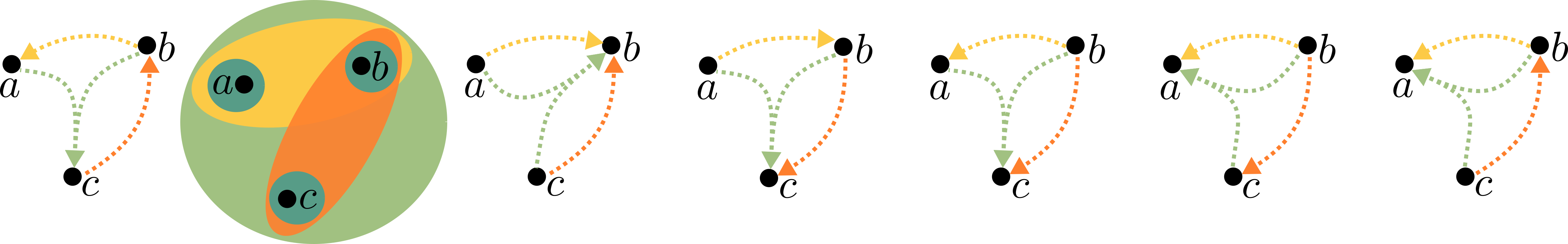}
 \caption{The hypergraph $h=\left(\left\{a,b,c\right\},\left\{\left\{a,b,c\right\},\left\{a,b\right\},\left\{b,c\right\},\left\{a\right\},\left\{b\right\},\left\{c\right\}\right\}\right)$ with a cyclic heading (left) and all its acyclic headings (right).}
 \label{abb:hg_def}
\end{figure}

The following  description of the vertices of the hypergraphic polytope in terms of acyclic orientations  plays a central role in the remainder of this paper
and is a particular instance of, e.g., \cite[Theorem~$2.18.$]{Benedetti2019}.
 \cref{prop:hgPconv} was stated without proof in  \cite{CardinalJean2018}.
 For convenience we give an elementary proof generalizing the proof idea for graphs presented in \cite{CardinalJean2018}.
\begin{prop}\label{prop:hgPconv}
 For a hypergraph $h=(I,E)$ the hypergraphic polytope $\Pol(h)$ can be described as
 \begin{equation}
  \Pol(h)=\conv\{\ \delta(\Or)\in \R I \ \colon\ \Or 
  \text{ is an \emph{acyclic} heading of }h \}
 \end{equation}
 where 
 \begin{equation}
  \delta(\Or)_i = \lvert \Or^{-1}(i)\rvert \quad \text{for }i\in I,
 \end{equation}
 i.e., $\delta (\Or)\in\R I$ is the vector of in-degrees of the nodes $i\in I$ 
 in the heading $\Or$.
\end{prop}
\begin{proof}
 Since the Minkowski sum of convex hulls of point sets is the same as the convex hull of the Minkowski sum of the points sets, we have
  \begin{equation}
  \Pol(h)=\sum_{e\in E}\conv\big({\{b_i\ \colon \ i\in e\}}\big)=\conv\left( \sum_{e\in E} \{b_i\ \colon \ i\in e\}\right).
 \end{equation}
 Every point in the convex hull on the right-hand side is 
 the vector of in-degrees of the nodes for some heading $\Or$. 
 Indeed, choosing some $b_i$ in every summand corresponds to choosing $i\in e$ as the head for the hyperedge $e$, and vice versa.
 It is left to show  that $\delta(\Or)$ is a vertex of $\Pol(h)$ if and only if the heading $\Or$ is acyclic.
 
 First, consider a heading containing an oriented cycle $e_1,\dots,e_\ell$. 
 Then
 \begin{equation}
  {\Or(e_1)\in e_2\setminus\Or(e_2), \dots, \Or(e_\ell)\in e_1\setminus\Or(e_1)}
 \end{equation}
 holds.
 We will construct new headings $\Or_1^*,\dots,\Or_\ell^*$ such that their vectors of in-degrees \linebreak
  $\delta(\Or_1^*),\dots,\delta(\Or_\ell^*)$ convex combine the vector of in-degrees $\delta(\Or)$ of the original heading $\Or$.
We define the new headings $\Or^*_j$ by changing the orientation of the
 hyperedge $e_j$ in the cycle, as depicted in \cref{fig:hgZykel}:
 \begin{equation}
\begin{split}
 \Or_1^*(e)\coloneqq\begin{cases}
               \Or(e_\ell) &\quad \text{if } e=e_1\\
               \Or(e) &\quad \text{otherwise} 
              \end{cases}
\quad\text{and for }j=2,\dots,\ell\, \quad
 \Or_j^*(e)\coloneqq\begin{cases}
               \Or(e_{j-1}) &\quad \text{if } e=e_j\\
               \Or(e) &\quad \text{otherwise.} 
              \end{cases}
\end{split}
\end{equation}
 \begin{figure}
 \centering
 \includegraphics[width=\textwidth]{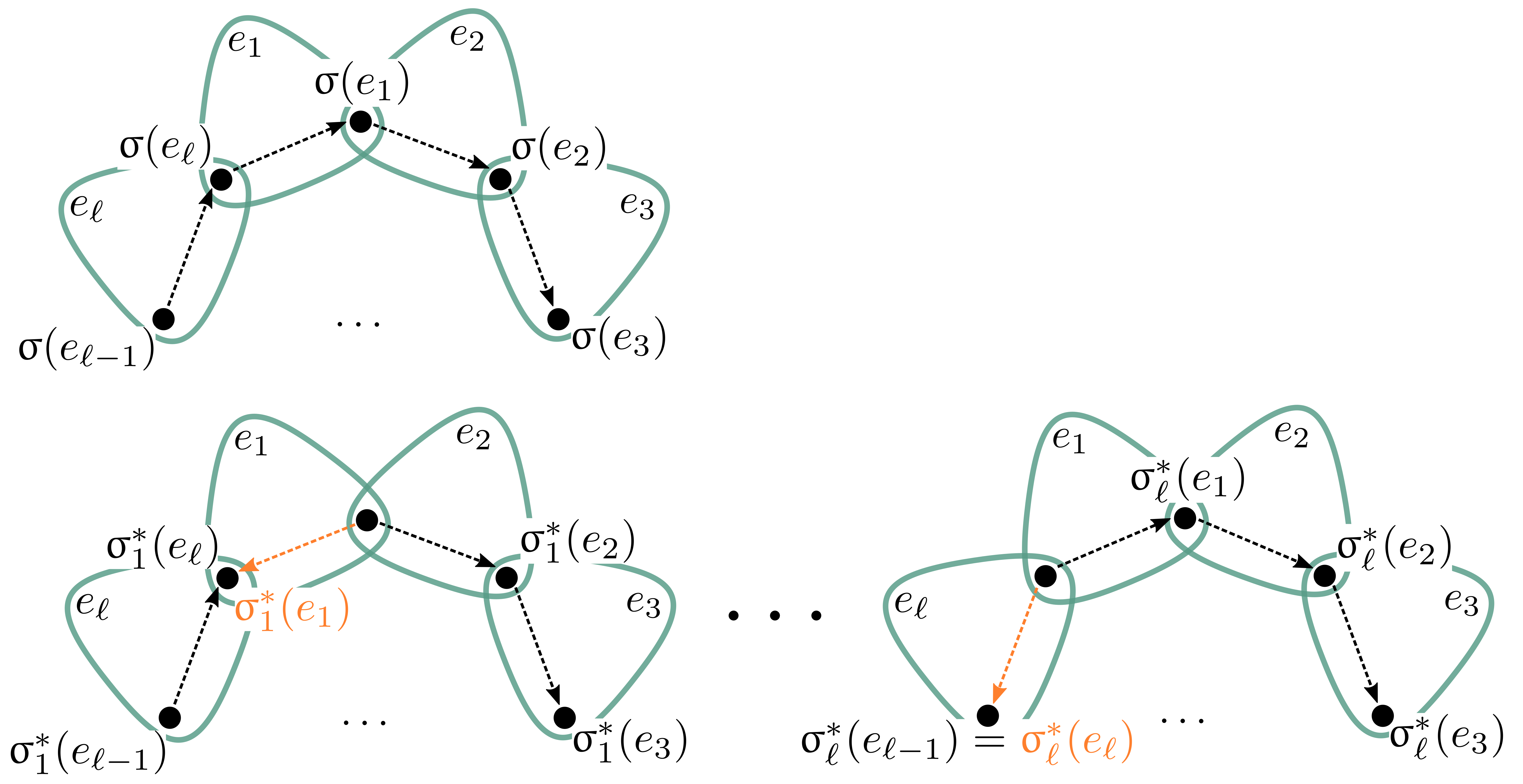}
 \caption{An oriented cycle $e_1,\dots,e_\ell$ with heading $\Or$ (top) and the 
 new headings $\Or_1^*,\dots,\Or_\ell^*$ on the edges $e_1,\dots,e_\ell$ (below).}
 \label{fig:hgZykel}
\end{figure}
 Then
 \begin{equation}
 \delta(\Or)=\sum_{j=1}^\ell \frac{1}{\ell}\delta(\Or_j^*).
\end{equation}
 Therefore, the vector of in-degrees $\delta(\Or)$ of a  heading $\Or$ containing a cycle cannot be a vertex.
 
 Now, let $\Or$ be an acyclic heading and let us assume there are headings 
 $\Or^*_1,\dots,\Or_l^*$ and scalars $0\leq\lambda_1,\dots, \lambda_l\in \R$ such that
\begin{equation}
 \delta(\Or)=\sum_{j=1}^\ell \lambda_j\delta(\Or_j^*)\quad\text{and}\quad\sum_{j=1}^\ell\lambda_i=1.
 \label{eq:conOr}
\end{equation}
First note that hyperedges $e$ with cardinality $\lvert e\rvert=1$ have only one possible
heading (the one choosing the only node in the hyperedge as head)
and those edges do not appear in oriented cycles. 
Hence they are irrelevant when it comes to deciding whether an heading is acyclic or not. 
Therefore we delete all singleton hyperedges and adjust the values in $\delta(\Or)$ as well as in $\delta(\Or^*_1),\dots,\delta(\Or_\ell^*)$.

 Since the heading $\Or$ is acyclic and we deleted all singleton hyperedges, there exists 
 at least one  source $s\in I$ with $\delta(\Or)_s=0$.
 From \cref{eq:conOr} it follows that $\delta(\Or_j^*)_s=0$  for all $j=1,\dots, \ell$.
 So, for the node $s$ the in-degree of all the headings is identical.
 We proceed by first deleting the source $s$ in all hyperedges,  then deleting all hyperedges $e$ with cardinality $\lvert e\rvert=1$,
 and adjusting the entries in $\delta(\Or)$,  $\delta(\Or^*_1),\dots,\delta(\Or_l^*)$.
 After finitely many iterations (the node set $I$ is finite) we get
 $\delta(\Or)_i=\delta(\Or_j^*)_i$ for every node $i\in I$ and all $j=1,\dots,\ell$ and 
 the in-degree vector $\delta(\Or)$ of the acyclic heading $\Or$ cannot be written as 
 a convex combination, that is, $\delta(\Or)$ is a vertex.
\end{proof}

 A \demph{coloring of a hypergraph} $h=(I,E)$ with $m$ colors is a map $c\colon I\to [m]$
 that assigns a color $c(i)\in[m]$ to every node $i\in I$.
 A node $i\in e\in E$ is called a \demph{maximal node} in the hyperedge~$e$ for the coloring $c$ if the color $c(i)$ is maximal among the colors in the hyperedge $e$,
 that is $c(i)=\max_{j\in e}c(j)$. The color $\max_{j\in e}c(j)$ is called the \demph{maximal color}.
 A coloring $c:I\to [m]$ of a hypergraph $h=(I,E)$ is called \demph{proper}
 if every hyperedge $e\in E$ contains a unique maximal node $i\in e$.
 This definition of a proper coloring is the same as, e.g., in \cite{Aval2020},
 but different from the ones in \cite{Erdoes1966, Bujtas2015, Breuer2012, Agnarsson2005}.
 A coloring $c\colon I\to [m]$ and a heading $\Or\colon E\to I$ of a hypergraph $h=(I, E)$
 are said to be \demph{compatible} if $c(\Or(e))=\max_{j\in e}c(j)$, i.e.,
 if the head $\Or(e)$ of a hyperedge $e$ has maximal color. 
See \cref{abb:hgColorComp} for some examples.
\begin{figure}
 \begin{subfigure}[t]{.245\textwidth}
 \centering
  \includegraphics[width=.8\linewidth]{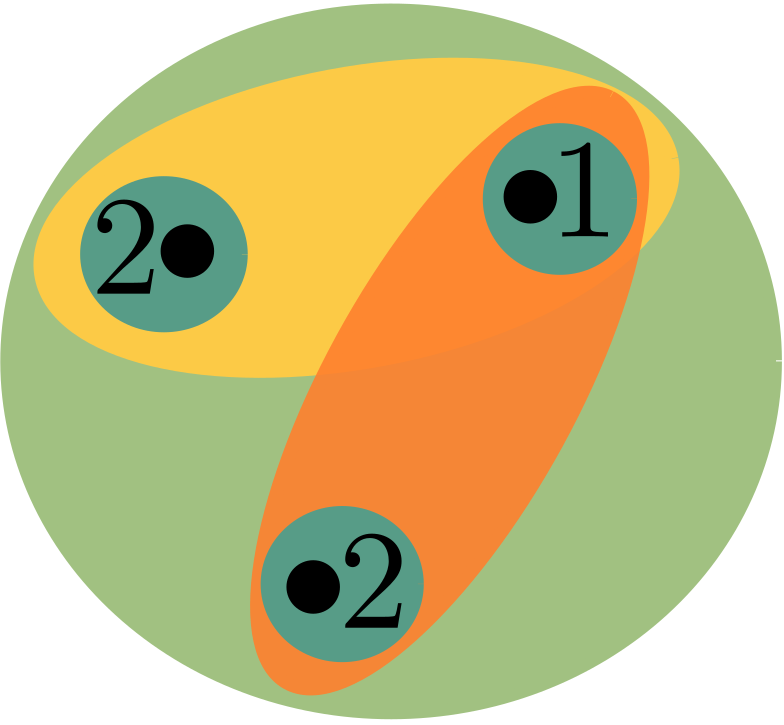}
  \caption{Not a proper coloring $c_1\colon\{a,b,c\}\to\{1,2\}$.}\label{abb:hgColor1}
 \end{subfigure}\hfill%
\begin{subfigure}[t]{.23\textwidth}
\centering
  \includegraphics[width=.8\linewidth]{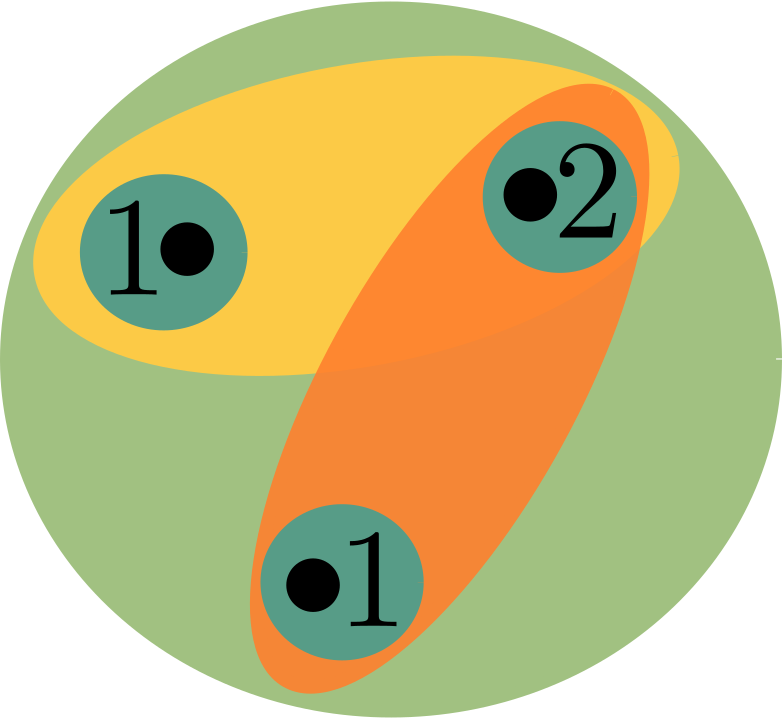}
  \caption{A proper coloring $c_2\colon\{a,b,c\}\to\{1,2\}$.}\label{abb:hgColor2}
 \end{subfigure}\hfill%
 \begin{subfigure}[t]{.23\textwidth}
 \centering
  \includegraphics[width=.8\linewidth]{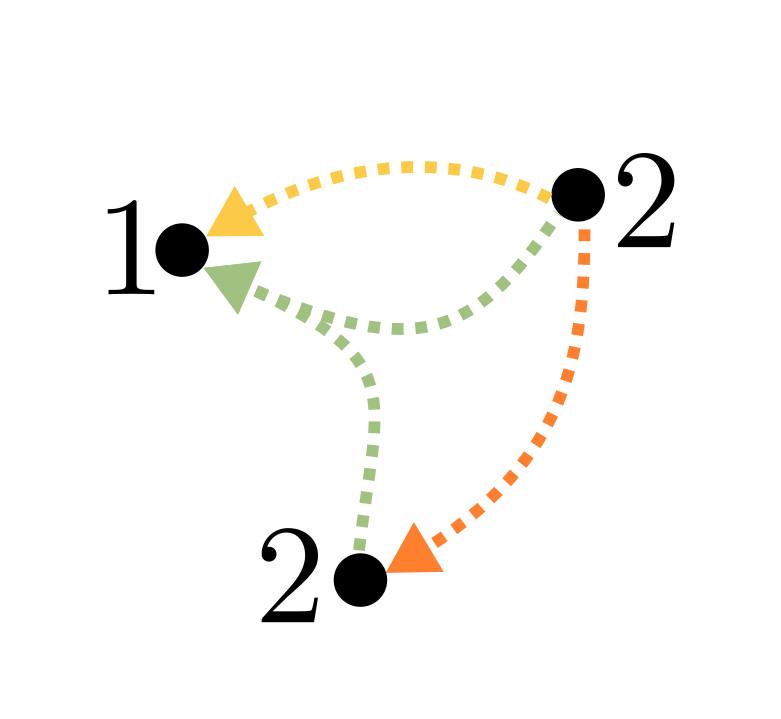}
  \caption{Incompatible heading and coloring.}\label{abb:hgComp1}
 \end{subfigure}\hfill%
 \begin{subfigure}[t]{.23\textwidth}
 \centering
  \includegraphics[width=.8\linewidth]{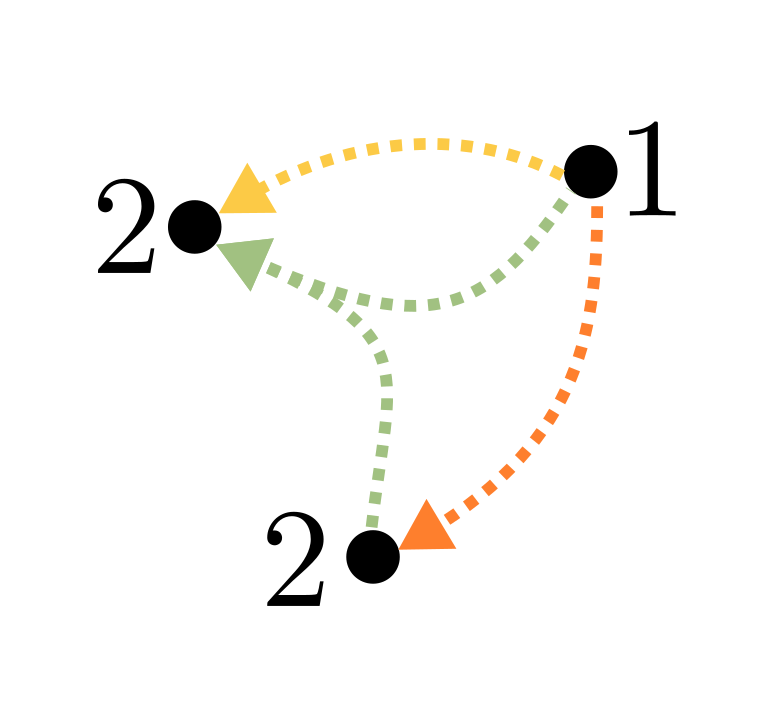}
  \caption{Compatible heading and coloring.}\label{abb:hgComp2}
 \end{subfigure}
 \caption{Hypergraph $h=\left(\left\{a,b,c\right\},\left\{\left\{a,b,c\right\},\left\{a,b\right\},\left\{b,c\right\},\left\{a\right\},\left\{b\right\},\left\{c\right\}\right\}\right)$ with colorings $c_i\colon\{a,b,c\}\to\{1,2\}$.}\label{abb:hgColorComp}
\end{figure}
\begin{rem}
 Considering usual graphs, the above definitions of (proper) colorings, (acyclic) headings and compatible pairs for hypergraphs specialize to those commonly used for graphs. In the same way the following \cref{prop:maxChrom} and \cref{prop:hgreci} generalize Stanley's reciprocity theorem for chromatic polynomials of graphs \cite{Stanley1973}.
\end{rem}
\begin{thm}[{\cite[Theorem 18]{Aval2020}}]\label{prop:maxChrom}
For a hypergraph $h=(I,E)$ with $\lvert I\rvert\eqqcolon d$ and a positive integer $m\in \Z_{>0}$,
\begin{equation}
 \chi_d(h)(m)\coloneqq \# (\text{proper colorings of $h$ with m colors})
\end{equation}
agrees with a polynomial in $m$ of degree $d$.
\end{thm}
\begin{proof}
Without loss of generality we assume $I=[d]$.
 For a hypergraph $h=(I,E)$ we consider  its corresponding hypergraphic polytope $\Pol(h)$ and since
 $\Pol(h)$ is a generalized permutahedron we can apply \cref{cor:genP}. 
 Hence we need to show
 \begin{equation}
  \# (\Pol(h)\text{-generic directions }  \y\in[m]^d)
	= \# (\text{ proper colorings of $h$ with $m$ colors}).
 \end{equation}
 We do so via a bijection. 
 For  $\y\in[m]^d$ we define the coloring $c_\y(i)\coloneqq y_i$ for $i=1,\dots,d$ 
 and vice versa, for a coloring $c\colon I\to[m]$ define $\y^c\in[m]^d$ by $y^c_i\coloneqq c(i)$.

 It is left to show that a direction $\y\in[m]^d$ is $\Pol(h)$-generic if and only if the coloring $c_y$ is proper.
 Recall $\y\in \R^I$ is $\Pol(h)$-generic if the maximal face $\left(\Pol(h)\right)_\y$ in direction $\y$ is a vertex.
 Linear functionals and Minkowski sums commute (see, e.g., \cite[Lemma 7.5.1]{Beck2018}), so 
 \begin{equation}\label{eq:MSy}
  \left(\Pol(h)\right)_\y = \bigg(\sum_{e\in E} \Delta_e\bigg)_\y =  \sum_{e\in E}(\Delta_e)_\y\, .
 \end{equation}
 Since the Minkowski sum is a point if and only if every summand is a point, 
the direction $\y$ is $\Pol(h)$-generic if and only if it is 
$\Delta_e$-generic for every hyperedge $e\in E$.
Finally, the direction $\y$ is $\Delta_e$-generic if and only if 
$(\Delta_e)_\y$ is a vertex.
Recall that $\Delta_e=\conv\{b_i\ \colon \ i\in e\}$ is the convex hull of standard basis vectors $b_i$, so
\begin{equation}\label{eq:convY}
 \left(\Delta_e \right)_\y = \conv\Big\{ b_i\ \colon\ i\in e,\  \y(i)=\max_{j\in e}\y(j) \Big\}\,.
\end{equation}
Therefore $\left(\Pol(h)\right)_\y$ is a vertex, if and only if for every hyperedge $e$ the direction $\y$ has a unique maximal value among the entries $\y(i)$ with $i\in e$.
The last statement is equivalent to the coloring $c_\y$ having a unique maximal node, i.e., being proper. In summary, for a positive integer $m\in\Z_{>0}$
\begin{equation}
 \begin{split}
 \chi_d(h)(m) &=\# \left(\text{proper colorings of $h$ with $m$ colors}\right) \\
	&= \# \left(\Pol(h)\text{-generic directions }  \y\in[m]^d \right) = \chi_d(\Pol(h))(m)
 \end{split}
\end{equation}
which is a polynomial in $m$ of degree $d$.
\end{proof}

\begin{thm}[{\cite[Theorem 24]{Aval2020}}]\label{prop:hgreci}
 Let $h=(I,E)$ be a hypergraph and $m\in\Z_{>0}$ a positive integer. Then
 \begin{equation}
 \begin{split}
    (-1)^d\chi_d(h)(-m)=\# (&\text{compatible  pairs of acyclic headings of $h$}\\
			&\text{ and colorings of $h$ with $m$ colors}).
 \end{split}
 \end{equation}
 In particular, the number of acyclic headings of $h$ equals $(-1)^d\chi_d(h)(-1)$.
\end{thm}
Note that the colorings do not need to be proper here.
\begin{proof}
 We follow the same idea as in the previous proof, that is, we use
 \begin{equation}
 \begin{split}
	(-1)^d\chi_d(h)(-m)=(-1)^d\chi_d(\Pol(h))(-m) 
	 = \sum_{\y\in[m]^d} \#  \left(\text{vertices of } \Pol(h)_\y\right)\,
 \end{split}
\end{equation}
and need to show
\begin{equation}
\begin{split}
  \sum_{\y\in[m]^d}  \# \left(\text{vertices of } \Pol(h)_\y\right)
  = &\sum_{c\text{  $m$-coloring of }h}  \# \left(\text{acyclic headings of $h$ compatible to  } c\right)\,.
\end{split}
\end{equation}
We use the same bijection between $m$-colorings $c_\y$ of $h$ and directions $\y^c\in[m]^d$ as above.
It is left show that for every direction $\y\in[m]^d$ the number of vertices of the maximal face 
$(\Pol(h))_\y$ in direction $\y$ 
equals the number of  acyclic headings of  $h$ compatible to the coloring $c_\y$ defined by the direction $\y$.
We compute the $\y$-maximum faces as in \cref{eq:MSy,eq:convY}:
\begin{equation}\label{eq:phmax}
 \Pol(h)_\y= \bigg(\sum_{e\in E}\Delta_e \bigg)_\y =\sum_{e\in E}(\Delta_e)_\y 
	=\sum_{e\in E}\conv\Big\{b_i\in\R I\ \colon\ i\in e,\ y(i)=\max_{j\in e}y(j)\Big\} \,.
\end{equation}
From \cref{eq:phmax} we can see that a vertex of $\Pol(h)_\y$ corresponds to choosing for every hyperedge $e\in E$ one of the  nodes $i\in e$ with maximal entry $\y(i)$, i.e., maximal color $c_\y(i)$. 
This is, by definition, the same as constructing a compatible heading for the coloring $c_\y$.
We know by \cref{prop:hgPconv} that vertices correspond to acyclic headings. 
Hence, vertices of $\Pol(h)_\y$ correspond to acyclic headings compatible to the coloring $c_\y$.
Vice versa, for a coloring $c$ the compatible acyclic headings are those with heads of hyperedges having a maximal coloring. 
That is, these acyclic headings correspond to those vertices, that are vertices of the maximum face $\Pol(h)_{\y^c}$ in direction~$\y^c$.
\end{proof}

\section*{Acknowledgments}
The author is extremely grateful to Matthias Beck for numerous enlightening conversations that, in particular, resulted in the idea for this paper, and for many useful hints, comments and suggestions during the process of writing this paper.
The author also wishes to thank Stefan Felsner for his support during and after writing her master thesis, as well as for introducing the author to various classes of polytopes and their interesting combinatorial properties. 
\cref{Assec:hg,appendix} were part of that master thesis.
The author would like to thank Thomas Zaslavsky for his nice idea to call the hypergraph orientations considered here \emph{headings} and for further helpful comments, as well as an anonymous referee for their useful suggestions.

\appendix
\section[Appendix]{Equivalent descriptions of permutahedra, generalized permutahedra, and hypergraphic polytopes}\label{appendix}

We compile some useful and nice information about permutahedra, generalized permutahedra and hypergraphic polytopes in this Appendix. This section is not original work, but the proof of \cref{thm:gpsf} below might be hard to find in the literature.
As explained in \Cref{ssec:notation} we will use the notations for $\R^d$ and $\R^I$ interchangeably.

Recall that we defined the standard permutahedron $\pi_d$ to be the convex hull of all the permutations of the point with entries $[d]$.
The facet description of the standard permutahedron is given by 
\begin{equation}
 \begin{split}
    \sum_{i=1}^d x_i&=d+\left(d-1\right)+\dots+1=\frac{d(d+1)}{2}\\
    \sum_{i\in T}x_i&\leq d+\left(d-1\right)+\dots+\left(d-\abs{T}+1\right) \quad \text{for all }T\subseteq [d].
 \end{split}
\end{equation}
Moreover, every face of the standard permutahedron can be described combinatorially by compositions, for details see, e.g., \cite[Section~4.1.]{AA17}.
The standard permutahedron can equivalently be described as the Minkowski sum of line segments:
 \begin{equation}
  \pi_d=\sum_{i<j} \Delta_{\{i,j\}},
 \end{equation}
where $\Delta_{\{i,j\}}\coloneqq \conv\{e_i,e_j\}$ and $e_i$ are standard basis vectors. This implies, in particular, that standard permutahedra are zonotopes.

Recall that generalized permutahedra are those polytopes that have a coarsening of the braid fan as normal fan.
Since the normal fan $\N(\Pol+\Q)$ of the Minkowski sum $\Pol+\Q$ of two polytopes $\Pol$ and $\Q$ is the common refinement of the two normal fans $\N(\Pol)$ and $\N(\Q)$ \cite[Proposition 7.12]{ziegler_lectures_1998},
generalized permutahedra are the (weak) Minkowski summands of standard permutahedra. That is, $\Pol\subset\R^d$ is a generalized permutahedron if and only if there exists a polytope $\Q\subset \R^d$ and a real scalar $\lambda>0$ such that $\Pol+\Q=\lambda \pi_{[d]}$.

One picturesque way of defining generalized permutahedra is by deforming standard permutahedra by parallel shifts of facets. 
This deformation maintains the normal fan until a face degenerates, i.e., at least two vertices are merged into one vertex. In that case the corresponding normal cones of the vertices are glued together. 
One example can be seen in \cref{fig:verallPerm}, where the top right edge degenerated and the two corresponding neighboring full-dimensional cones were combined.
A formal description of these deformations and a detailed proof of equivalence can be found in \cite[Appendix]{postnikov_faces_2006}.

Finally, generalized permutahedra can be uniquely described as the base polytopes of submodular functions $z\colon 2^I \to \R$ with $z(\varnothing)=0$ (\cref{thm:gpsf}).
See, e.g., \cite[Theorem 3.11 and 3.17]{castillo_deformation_2020}.
For the sake of completeness and the convenience of the reader we include a self-contained proof of the well-known equivalence of the definitions of generalized permutahedra through braid fan coarsenings and submodular functions. 
 A set function $z\colon 2^I\to \R$ is called \demph{submodular} if for all $A,B\subseteq I$
  \begin{equation}
	z(A)+z(B)\geq z(A\cup B)+z(A\cap B)\,.
 \end{equation} 
 We define the \demph{base polytope} $\Po(z)$ of a submodular function $z\colon 2^I\to \R$ by
 \begin{equation}
  \Po(z)\coloneqq\big\{x\in \R I\ \colon\ \sum_{i\in I}x_i=z(I)\ \text{ and }\ 
  \sum_{i\in A}x_i \leq z(A) \ \text{ for all }A\subseteq I \big\}\,.
 \end{equation}
 To simplify the proof of \cref{thm:gpsf} we will use the following notation
 \begin{equation}
  x(A)\coloneqq \sum_{i\in A}x_i \qquad\text{ for }A\subseteq I \,.
 \end{equation}
With that notation at hand we can write the definition of the base polytopes as 
 \begin{equation}
  \Po(z)= \big\{x\in \R I\ \colon\ x(I)=z(I)\ \text{ and }\ 
  x(A) \leq z(A) \ \text{ for all }A\subseteq I\big\}\,.
 \end{equation}
As mentioned above  the standard permutahedron $\pi_I$ is the base polytope of the submodular function 
\begin{equation}
 z(A)\coloneqq\lvert I\rvert+ (\lvert I\rvert-1)+\dots+ (\lvert I\rvert-\lvert A\rvert+1)\,.
\end{equation}

\begin{thm}[]\label{thm:gpsf}
 A polytope $\Pol$ is a generalized permutahedron if and only if it is the base polytope $\Po(z)$ of a submodular function $z\colon 2^I\to \R$ with $z(\varnothing)=0$.
\end{thm}
Before we start proving this theorem, we give a description of the faces in terms of set compositions.
 A \demph{composition} of a finite set $I$ is an ordered sequence $(T_1,\dots,T_k)$ of disjoint non-empty subsets $T_i\subseteq I$ such that $I=T_1\uplus\dots\uplus T_k$.
Let $\mathbf{1}_T$ for some $T\subseteq I$ be the $0/1$-vector with entries equal to one  for indices in the subset $T$ and zero otherwise.

\begin{lem}\label{lem:braid}
 The faces of the braid arrangement $\B_I$, also called \demph{braid cones}, can be described uniquely by compositions $I=T_1\uplus\dots\uplus T_k$:
  \begin{equation}
   \begin{split}
    \B_{T_1,\dots,T_k}\coloneqq& \big\{y\in\R^I\ \colon\ y(i)=y(j)\text{ for all }i,j\in T_a,\ 
	y(i)\geq y(j)\text{ for }i\in T_a, j\in T_b \text{ and }a<b\big\}\\
    =&\cone\{\mathbf{1}_{T_1}, \mathbf{1}_{T_1\cup T_2},\dots,
			\mathbf{1}_{T_1\cup\dots\cup T_{k-1}} \} + \spn\{\mathbf{1}_{I}\}
   \end{split}
\end{equation}
with
\begin{equation}
 \dim \B_{T_1,\dots,T_k}=k\,,
\end{equation}
where $\mathbf{1}_T$ for some subset $T\subseteq I$ is the $0/1$-vector with entries equal to one  for indices in the subset $T$ and zero otherwise.
\end{lem}

\begin{proof}[Proof of \cref{thm:gpsf}]
For a submodular function $z\colon 2^I\to  \R$ we show that $\Po(z)$ is a generalized permutahedron by showing that every braid cone $\B_{T_1,\dots,T_k}\subset\R^I$ is contained in a normal cone of $\Po(z)$.
Since $\Po(z)$ is contained in the hyperplane $\{x\in \R I\ \colon\  x(I)=z(I)\}$ the normal cone $N_{\Po(z)}(\Po(z))$ contains the line spanned by  $\mathbf{1}_I \in\R I$, hence every  normal cone of $\Po(z)$ contains that line.

The following part of the proof relies on \cite{Fujishige1983}. 
Fujishige and Tomizawa show under which conditions a greedy-like algorithm gives an optimal solution in the base polytope of a submodular functions on a general distributive lattice.
We adapt the proof to our special case.

Let $\B_{T_1,\dots,T_k}\subset\R^I$ a braid cone.
Choose a maximal chain $\mathcal{C}\colon\emptyset= C_0\subset\dots\subset C_n=I$ in the boolean lattice $2^I$ such that $T_1,T_1\sqcup T_2,\dots,T_1\sqcup\dots\sqcup T_k$ are sets in the chain $\mathcal{C}$.
Then
\begin{equation}
 \lvert C_j\setminus C_{j-1}\rvert =1 
\end{equation}
for $j=1,\dots,n\coloneqq\lvert I\rvert$ 
and we define a linear ordering on $I$ by $i_j\coloneqq C_j\setminus C_{j-1}\in I$ for $j=1,\dots,n$.
Now, consider the point $\tilde{x}\in\R I $ defined by
\begin{equation}
\label{eq:xtilde}
 \tilde{x}_{i_j}\coloneqq z(C_j)-z(C_{j-1}) \quad \text{für }j=1,\dots,n.
\end{equation}
We will show   
\begin{enumerate}[nolistsep]
 \item that $\tilde{x}(C_j)=z(C_j)$ for $j=1,\dots,n$,
  and that the point $\tilde{x}$ lies in $\Po(z)$,
 \item that $\tilde{x}$ is maximal for all directions in the braid cone $\B_{T_1,\dots,T_k}$.
\end{enumerate}
Then it follows that the braid cone $\B_{T_1,\dots,T_k}$ is contained in the normal cone $N_{\Po(z)}(F)$, where $F$ is a face containing $\tilde{x}$.

For $j=1,\dots,n$ we compute
\begin{equation}
 \tilde{x}(C_j)=\sum_{l=1}^j \tilde{x}_{i_l}=\sum_{l=1}^j \left(z(C_l)-z(C_{l-1})\right)=z(C_j),
\end{equation}
in particular, $\tilde{x}(I)=z(I)$. 
We show by induction on the cardinality $\lvert A\rvert$ of a subset $A\subseteq I$ that $\tilde{x}(A)\leq z(A)$.
For the empty set we have $0=\tilde{x}(\varnothing)=z(\varnothing)$. 
For an arbitrary set $A\subseteq I$ let $j^*$ be the minimal index such that $A\subseteq C_{j^*}$ and define the element $i^*\coloneqq A\setminus C_{j^*-1} \in I$.
We compute using the induction hypothesis, \cref{eq:xtilde}, and submodularity of $z$ together with $A\setminus\{i^*\}= A \cap C_{j^*-1}$ and $C_{j^*}=A\cup C_{j^*-1}$:
\begin{equation}
\begin{split}
 \tilde{x}(A)\ &=\ \tilde{x}(\{i^*\})+\tilde{x}(A\setminus\{i^*\})\ \leq\  \tilde{x}(\{i^*\}) 
+z(A\setminus\{i^*\})\\
&=\ z(C_{j^*})-z(C_{j^*-1})+z(A\setminus\{i^*\})\ \leq\ z(A)\,.
\end{split}
\end{equation}
Hence, $\tilde{x}\in\Po(z)$.

Now, choose an arbitrary direction $\y\in\B_{T_1,\dots,T_k}$.
By Lemma~\ref{lem:braid} $y(i)=y(i')$ for $i,i'\in T_l$ so we can set $\hat{y}_l\coloneqq y(i) $ for $i\in T_l$ and $l=1,\dots,k$. Moreover, $\hat{y}_l\geq \hat{y}_{l+1}$. 
For a point $x\in\Po(z)$ compute:
\begin{equation}\label{eq:yx}
\begin{split}
 \y(\tilde{x})-\y(x)&=\sum_{i\in I}\tilde{x}_i\y(i)-\sum_{i\in I}x_i\y(i)
  =\sum_{l=1}^k \hat{y}_l\big(\tilde{x}(T_l)-x(T_l)\big) \\
  &= \sum_{l=1}^k \Big(\hat{y}_l \big(\tilde{x}(T_1\sqcup\dots\sqcup T_l)-x(T_1\sqcup\dots\sqcup T_l)\big)\\
	&\hspace{28ex}-\hat{y}_l \big(\tilde{x}(T_1\sqcup\dots\sqcup T_{l-1})-x(T_1\sqcup\dots\sqcup T_{l-1})\big)\Big)\\
 &= \sum_{l=1}^{k-1} (\hat{y}_l-\hat{y}_{l+1}) \big(\tilde{x}(T_1\sqcup\dots\sqcup T_l)-x(T_1\sqcup\dots\sqcup T_l)\big)
	+\hat{y}_k \big(\tilde{x}(I)-x(I)\big)\\
  &= \sum_{l=1}^{k-1} (\underbrace{\hat{y}_l-\hat{y}_{l+1}}_{\geq0}) 
  \big(\underbrace{z(T_1\sqcup\dots\sqcup T_l)-x(T_1\sqcup\dots\sqcup T_l)}_{\geq0}\big)\geq 0,
\end{split}
\end{equation}
where we use in the last equality, that the sets $T_1,T_1\sqcup T_2,\dots,T_1\sqcup\dots\sqcup T_k$ are contained in the chain $\mathcal{C}\colon\emptyset= C_0\subset\dots\subset C_n=I$ and that we already know $\tilde{x}(C_j)=z(C_j)$ for $j=1,\dots,n$.
Since the computation in \ref{eq:yx} is independent from the actual values of the direction $\y\in\B_{T_1,\dots,T_k}$, the inequality $\y(\tilde{x})\geq \y(x)$ holds for every direction $\y\in\B_{T_1,\dots,T_k}$.
So the braid cone $\B_{T_1,\dots,T_k}$ is contained in the normal cone $N_{\Po(z)}(F)$, where $F$ is a face containing $\tilde{x}$. Hence, $\Po(z)$ is a generalized permutahedron.

For the opposite implication let $\Pol$ be a generalized permutahedron.
We will define a submodular function $z_{\Pol}$ and show that  $\Pol=\Po(z_{\Pol})$.
Since the generalized permutahedron $\Pol$ is contained in the hyperplane with constant coordinate sum, the following set function is well defined:
\begin{equation}
 \begin{split}
  z_\Pol(I)&\coloneqq\sum_{i\in I} x_i \quad\text{ for }x\in\Pol\\
  z_\Pol(A)&\coloneqq \max_{x\in\Pol}\bigg(\sum_{i\in A}x_i\bigg) \quad \text{ for }A\subseteq I\,.
 \end{split}
\end{equation}
We can immediately deduce that $z(\varnothing)=0$ and $\Pol\subseteq\Po(z_\Pol)$.

First we show that $z_{\Pol}$ is submodular. 
For arbitrary $A, B\subseteq I$ find a chain $\mathcal{C}\colon \emptyset=C_0\subset C_1\subset\dots\subset C_k=I$ in the Boolean lattice $2^I$ that contains  $A\cap B$ and $A\cup B$.
We set $T_i\coloneqq C_i\setminus C_{i-1}$ for $i=1,\dots,k$ and consider the braid cone $\B_{T_1,\dots,T_k}=\cone\{\mathbf{1}_{T_1},\dots,\mathbf{1}_{T_1\sqcup\dots\sqcup T_{k-1}}\}
+\spn\{\mathbf{1}_I\}$.
Then there exists a face $F$ of $\Pol$ such that the normal cone $N_\Pol(F)$ contains the braid cone $\B_{T_1,\dots,T_k}$ and in particular every point $x\in F$ is maximal in the direction $\mathbf{1}_{A\cap B},\mathbf{1}_{A\cup B}\in \B_{T_1,\dots,T_k}$. 
Then,
\begin{equation}
 z_\Pol(A)+z_\Pol(B)\geq x(A)+x(B)=x(A\cup B)+x(A\cap B)=z_\Pol(A\cup B)+z_\Pol(A\cap B)
\end{equation}
and $z_\Pol$ is submodular.

Now it is left to show that  $\Pol\supseteq\Po(z_\Pol)$.
The main idea for this part of the proof can be found in \cite{Derksen2010}.
For the sake of a contradiction, let us assume there is a point $u\in \Po(z_\Pol)\setminus\Pol$.
Then there exists a separating hyperplane $H_{\vt,c}\coloneqq\{x\in\R I\ \colon\ \vt(x)=c\}$ such that
\begin{equation}
 \vt( u)=\sum_{i\in I}\vt_i u_i > c\quad\text{and}\quad \vt( p)=\sum_{i\in I}\vt_i p_i\leq c
	\quad\text{for all }p\in\Pol
\end{equation}
Now choose a braid cone $\B_{T_1,\dots,T_k}$ such that $t\in\B_{T_1,\dots,T_k}$ and set again $\hat{t}_l\coloneqq t_i$  for $i\in T_l$, $l=1,\dots,k$.
For points $q$ in the $t$-maximal face $F\coloneqq \Pol_t$ we know by the definition of $z$ that ${q(T_1\cup\dots\cup T_l)=z(T_1\cup\dots\cup T_l)}$ for $l=1,\dots,k$. 
Using telescoping sums we compute
\begin{equation}
\begin{split}
 t( u)=\sum_{i\in I}t_i u_i > c \geq t\cdot q&=\sum_{i\in I}t_i q_i
 =\sum_{l=1}^k \hat{t}_jq(T_l)\\
 &= \hat{t}_k q(T_1\cup\dots\cup T_k)+\sum_{l=k-1}^1 (\hat{t}_l- \hat{t}_{l+1})q(T_1\cup\dots\cup T_l)\\
 &= \hat{t}_k z(T_1\cup\dots\cup T_k)+\sum_{l=k-1}^1 (\hat{t}_l- \hat{t}_{l+1})
	z(T_1\cup\dots\cup T_l)\\
 &\geq \hat{t}_k u(T_1\cup\dots\cup T_k)+\sum_{l=k-1}^1 (\hat{t}_l- 
\hat{t}_{l+1})u(T_1\cup\dots\cup T_l)\\
 &=\sum_{l=1}^k \hat{t}_l u(T_l) =\sum_{i\in I}t_i u_i =  t( u).
 \end{split}
\end{equation}
This is a contradiction and completes the proof.
\end{proof}

Recall that the hypergraphic polytope $\Pol(h)\subset \R I$ of a hypergraph $h=(I,E)$ is defined  as 
\begin{equation}
 \Pol(h)=\sum_{e\in E}\Delta_{e} \ \subset \R I
\end{equation}
where
 \begin{equation}
  \Delta_{e}= \conv\{b_i\ \colon\ i\in e\},\quad 
  \text{for a hyperedge }e\subseteq I\, 
 \end{equation}
 and $b_i$ are the basis vectors for $i\in I$.
\begin{prop}[{\cite[Proposition~6.3.]{postnikov_permutohedra_2009}}]\label{cor:hgPgP}
For a hypergraph $h=(I,E)$ and its hypergraphic polytope $\Pol(h)$, the function $z\colon 2^I\to \R$ defined by
\begin{equation}
 z(T)\coloneqq\sum_{\substack{e\in E\\ e\cap T\neq \varnothing} }1=\#(\text{hyperedges in $h$ that intersect }T) \quad \text{for }T\subseteq I
\end{equation}
is a submodular function with $z(\varnothing)=0$ and 
\begin{equation}
 \Pol(h)=\Big\{ x\in\R I\ \colon \sum_{i\in I} x_i=z(I) \quad \text{and} \quad  \sum_{i\in T} x_i\leq z(T)\quad \text{ for } T\subseteq I \Big\}\,.
\end{equation}
 Hence, hypergraphic polytopes  are generalized permutahedra and in bijection with hypergraphs.
\end{prop}
\begin{rem}\label{rem:test}
Postnikov uses a different convention for the facet description of a generalized permutahedron:
\begin{equation}
 \Pol(z)\coloneqq \Big\{(t_1,\dots,t_d)\in \R^d\ \colon \sum_{i=1}^d t_i = z([d])\,,\ \sum_{i\in J} t_i\geq z(J)\,, \text{ for  }J\subseteq [d] \Big\}\,.
\end{equation}
This results in a differing formulation of \cref{cor:hgPgP} which is nevertheless equivalent. 
\end{rem}
For an interesting characterization when a submodular function gives rise to a hypergraphic polytope see \cite[Proposition~19.4.]{AA17}.

\bibliographystyle{alpha}
\bibliography{piop}

\begin{thebibliography}{ACEP20}

\bibitem[AA17]{AA17}
Marcelo Aguiar and Federico Ardila.
\newblock Hopf monoids and generalized permutahedra.
\newblock 2017.
\newblock arXiv:1709.07504.

\bibitem[ACEP20]{ardila_coxeter_2020}
Federico Ardila, Federico Castillo, Christopher Eur, and Alexander Postnikov.
\newblock Coxeter submodular functions and deformations of {Coxeter}
  permutahedra.
\newblock {\em Advances in Mathematics}, 365, 2020.

\bibitem[Agn17]{agnarsson_special_2017}
Geir Agnarsson.
\newblock On a special class of hyper-permutahedra.
\newblock {\em Electron. J. Combin.}, 24(3):article P3.46, 25 pp, 2017.

\bibitem[AH05]{Agnarsson2005}
Geir Agnarsson and Magnús~M. Halldórsson.
\newblock Strong colorings of hypergraphs.
\newblock In Giuseppe Persiano and Roberto Solis-Oba, editors, {\em
  Approximation and {Online} {Algorithms}}, pages 253--266, Berlin, Heidelberg,
  2005. Springer.

\bibitem[AKT20]{Aval2020}
Jean-Christophe Aval, Théo Karaboghossian, and Adrian Tanasa.
\newblock The {Hopf} monoid of hypergraphs and its sub-monoids: Basic invariant
  and reciprocity theorem.
\newblock {\em Electron. J. Combin.}, pages article P1.34, pp23, 2020.

\bibitem[AM10]{aguiar_monoidal_2010}
Marcelo Aguiar and Swapneel~Arvind Mahajan.
\newblock {\em Monoidal functors, species, and {Hopf} algebras}.
\newblock American Mathematical Society, Providence, R.I., 2010.
\newblock OCLC: 989866302.

\bibitem[BBM19]{Benedetti2019}
Carolina Benedetti, Nantel Bergeron, and John Machacek.
\newblock Hypergraphic polytopes: Combinatorial properties and antipode.
\newblock {\em J. Comb.}, 10(3):515--544, 2019.

\bibitem[BDK12]{Breuer2012}
Felix Breuer, Aaron Dall, and Martina Kubitzke.
\newblock Hypergraph coloring complexes.
\newblock {\em Discrete Math.}, 312(16):2407--2420, 2012.

\bibitem[BJR09]{Billera2009}
Louis~J. Billera, Ning Jia, and Victor Reiner.
\newblock A quasisymmetric function for matroids.
\newblock {\em Eur. J. Combin.}, 30(8):1727--1757, 2009.

\bibitem[BS18]{Beck2018}
Matthias Beck and Raman Sanyal.
\newblock {\em Combinatorial {R}eciprocity {T}heorems: An {I}nvitation To
  {Enumerative} {Geometric} {Combinatorics}}.
\newblock American Mathematical Society, Providence, R.I., 2018.

\bibitem[BTV15]{Bujtas2015}
Csilla Bujtás, Zsolt Tuza, and Vitaly Voloshin.
\newblock Hypergraph colouring.
\newblock In Lowell~W. Beineke and Robin~J. Wilson, editors, {\em Topics in
  Chromatic Graph Theory}, pages 230--254. Cambridge University Press,
  Cambridge, 2015.

\bibitem[BZ06a]{Beck2006b}
Matthias Beck and Thomas Zaslavsky.
\newblock An enumerative geometry for magic and magilatin labellings.
\newblock {\em Ann. Comb.}, 10(4):395--413, 2006.

\bibitem[BZ06b]{Beck2006}
Matthias Beck and Thomas Zaslavsky.
\newblock Inside-out polytopes.
\newblock {\em Adv. Math.}, 205(1):134--162, 2006.

\bibitem[BZ06c]{Beck2006a}
Matthias Beck and Thomas Zaslavsky.
\newblock The number of nowhere-zero flows on graphs and signed graphs.
\newblock {\em J. Combin. Theory Ser. B}, 96(6):901--918, 2006.

\bibitem[BZ10]{Beck2010}
Matthias Beck and Thomas Zaslavsky.
\newblock Six little squares and how their numbers grow.
\newblock 2010.
\newblock arXiv:1004.0282.

\bibitem[CF18]{CardinalJean2018}
Jean Cardinal and Stefan Felsner.
\newblock Notes on {Hypergraphic} {Polytopes}.
\newblock Unpublished, 2018.

\bibitem[CL20]{castillo_deformation_2020}
Federico Castillo and Fu~Liu.
\newblock Deformation cones of nested braid fans.
\newblock {\em Int. Math. Res. Notices}, 2020.

\bibitem[DF10]{Derksen2010}
Harm Derksen and Alex Fink.
\newblock Valuative invariants for polymatroids.
\newblock {\em Adv. Math.}, 225(4):1840--1892, 2010.

\bibitem[EH66]{Erdoes1966}
P.~Erdős and A.~Hajnal.
\newblock On chromatic number of graphs and set-systems.
\newblock {\em Acta Math Acad Sci H}, 17(1-2):61--99, 1966.

\bibitem[Ehr62]{ehrhart_sur_1962}
Eugène Ehrhart.
\newblock Sur les polyèdres rationnels homothétiques à $n$ dimensions.
\newblock {\em C. R. Hebd. Seances Acad. Sci.}, 254:616--618, 1962.

\bibitem[FT83]{Fujishige1983}
Satoru Fujishige and Nobuaki Tomizawa.
\newblock A note on submodular functions on distributive lattices.
\newblock {\em J. Oper. Res. Soc. Japan}, 26(4):309--318, 1983.

\bibitem[Gru03]{Gruenbaum2003}
Branko Gr\"unbaum.
\newblock {\em Convex Polytopes}.
\newblock Springer New York, New York, 2003.

\bibitem[Kar20]{karaboghossian_polynomial_2020}
Théo Karaboghossian.
\newblock {\em Polynomial invariants and algebraic structures of combinatorial
  objects}.
\newblock phdthesis, Université de Bordeaux, October 2020.

\bibitem[Kar22]{karaboghossian_combinatorial_2022}
Théo Karaboghossian.
\newblock Combinatorial expressions of {Hopf} polynomial invariants.
\newblock 2022.
\newblock arXiv: 2203.03947.

\bibitem[Mac71]{macdonald_polynomials_1971}
Ian~G. Macdonald.
\newblock Polynomials associated with finite cell-complexes.
\newblock {\em J. London Math. Soc.}, s2-4(1):181--192, 1971.

\bibitem[Pos09]{postnikov_permutohedra_2009}
Alexander Postnikov.
\newblock Permutohedra, associahedra, and beyond.
\newblock {\em Int. Math. Res. Notices}, (6):1026--1106, 2009.

\bibitem[PRW08]{postnikov_faces_2006}
Alex Postnikov, Victor Reiner, and Lauren Williams.
\newblock Faces of generalized permutohedra.
\newblock {\em Doc. Math.}, 13:207--273, 2008.

\bibitem[RR12]{Reff2012}
Nathan Reff and Lucas~J. Rusnak.
\newblock An oriented hypergraphic approach to algebraic graph theory.
\newblock {\em Linear Algebra Appl.}, 437(9):2262--2270, 2012.

\bibitem[Rus13]{Rusnak2013}
Lucas~J. Rusnak.
\newblock Oriented hypergraphs: Introduction and balance.
\newblock {\em Electron. J. Combin.}, 20(3):article P48, 29 pp, 2013.

\bibitem[Sta73]{Stanley1973}
Richard~P. Stanley.
\newblock Acyclic orientations of graphs.
\newblock {\em Discrete Math.}, 5(2):171--178, 1973.

\bibitem[Sta74]{Stanley1974}
Richard~P. Stanley.
\newblock Combinatorial reciprocity theorems.
\newblock {\em Adv. Math.}, 14(2):194--253, 1974.

\bibitem[Sta86]{Stanley1986}
Richard~P. Stanley.
\newblock Two poset polytopes.
\newblock {\em Discrete Comput. Geom.}, 1(1):9--23, 1986.

\bibitem[Sta07]{Stanley2007}
Richard~P. Stanley.
\newblock An introduction to hyperplane arrangments.
\newblock In Ezra Miller, {Reiner, Victor}, and {Sturmfels, Bernd}, editors,
  {\em Geometric Combinatorics}, pages 389--496. American MathSoc, Providence,
  R.I., 2007.

\bibitem[Zas91]{Zaslavsky1991}
Thomas Zaslavsky.
\newblock Orientation of signed graphs.
\newblock {\em Eur. J. Combin.}, 12(4):361--375, 1991.

\bibitem[Zie98]{ziegler_lectures_1998}
Günter~M. Ziegler.
\newblock {\em Lectures on Polytopes}.
\newblock Springer, New York, 1998.

\end{thebibliography}

\end{document}